\documentclass[11pt, reqno]{amsart}

\usepackage{latexsym}
\usepackage{amsfonts, amsmath, amscd}
\usepackage{amssymb}
\usepackage{xypic}
\usepackage{graphicx,psfrag}
\usepackage{epstopdf}

\usepackage{color} 
\usepackage{enumerate}
\usepackage{verbatim}
\usepackage{amsthm}    
\usepackage{hyperref}
 \hypersetup{colorlinks=false}

\usepackage{pstool}
\usepackage{tikz-cd}
\usepackage{slashed}

\numberwithin{equation}{section}

\newtheorem{theorem}{Theorem}[section]
\newtheorem{prop}[theorem]{Proposition}

\newtheorem{lemma}[theorem]{Lemma}
\newtheorem{cor}[theorem]{Corollary}

\newtheorem{defn}[theorem]{Definition}

\newtheorem{remark}[theorem]{Remark}
\newenvironment{rem}{\begin{remark}\rm}{\end{remark}}
\newtheorem{example}[theorem]{Example}

\newtheorem*{spectrumseparationtheorem*}{Spectral Separation Theorem}
\newtheorem*{indexlocalizationtheorem*}{Index Localization Theorem}

\newcounter{nmdthmcnt}

\def\bear{\begin{eqnarray}}
\def\eear{\end{eqnarray}}

\def\a{\alpha}
\def\b{\beta}

\def\phi{\varphi}

\def\Bbb{\mathbb}

\def\Z{{ \Bbb Z}}
\def\R{{ \Bbb R}}

\def\cx{{ \Bbb C}}

\def\cal{\mathcal}
\def\A{{\mathcal A}}

\def\D{{\mathcal D}}

\def\K{{\mathcal K}}
\def\L{{\mathcal L}}

\def\N{\mathcal{N}}

\def\tr{\mbox{tr}}

\def\cok{\mathrm{coker\,}}
\def\ind{\mathrm{index\,}}

\def\End{\mathrm{End}}
\def\Hom{\mathrm{Hom}}
\def\Isom{\mathrm{Isom}}

\def\Spin{\mathrm{Spin}}

\def\span{\mathrm{span}}
\def\Ad{\mathrm{Ad}} 
\def\Sym{\mathrm{Sym}}

\def\ev{\mathrm{ev}}
 \def\del{\overline \partial}
\newcommand{\SM}[1]{%
  \ensuremath{\Bigl(\negthinspace\begin{smallmatrix}#1\end{smallmatrix}\Bigr)}}

\title[Spinor Pairs and the  Concentration Principle ]{Spinor Pairs and the  Concentration Principle for Dirac operators}
\author[M. Maridakis ]{Manousos Maridakis}
\address{Hill Center Department of  Mathematics, Rutgers University}
\email{mmanos@math.rutgers.edu}
\begin{document}


\vskip.15in

\begin{abstract}
We study perturbed Dirac operators of the form $ D_s= D + s{\cal A} :\Gamma(E)\rightarrow \Gamma(F)$ over a compact Riemannian manifold $(X, g)$ with symbol $c$ and special bundle maps ${\cal A} :  E\rightarrow F$ for $s>>0$. Under a simple algebraic criterion on the pair $(c, {\cal A})$, solutions of $D_s\psi=0$ concentrate as $s\to\infty$ around the singular set $Z_\A\subset X$ of ${\cal A}$.  We give many examples, the most interesting ones arising from a general ``spinor pair'' construction.
 \end{abstract}
\maketitle

 \vspace{5mm}

\setcounter{equation}{0}
\section{Introduction}
\bigskip

 Given a first order elliptic operator $D$ with principal symbol $\sigma$, one  can look for zeroth-order perturbations $\A$  such that all  finite-energy solutions of the equation
\begin{equation}
\label{m1}
D_s\psi\ =\ (D+s\A)\psi\ =\  0
\end{equation}
increasingly concentrate along submanifolds $Z_\ell$ as $s\to\infty$. There are several examples of this in the literature, the most well-known occurring in Witten's approach to Morse Theory. The aim of this paper is to find a general setting for such concentration phenomenon and construct new examples.

 We start with a simple criterion that insures concentration: $D + s{\cal A}$ localizes if 
\begin{eqnarray}
\label{co}
\A^*\circ \sigma(u) = \sigma(u)^*\circ \A  \qquad \mbox{for every $u\in T^*X$.}
\end{eqnarray}
This  algebraic  condition implies the analytic fact that solutions  concentrate (in the precise sense of Proposition~\ref{concentrationProp}).

 After summarizing known examples, this paper  describes a general method, based on ``spinor pairs'', that yields new examples. In dimensions two and four, this method produces operators that are closely related to the linearized vortex equations on a Riemann surface,  and to the Seiberg-Witten equations on a 4-manifold.   In both cases,  concentration occurs along submanifolds defined by the zeros of a spinor.

 The concentration condition (\ref{co}) was  previously found by I. Prokhorenkov and K.~Richardson \cite{pr}. They classified the complex linear perturbations $\A$ that satisfy \eqref{co} but found few examples, all of which concentrate at points.  Their list of examples does {\em  not}  include most of the examples given here because they assumed that $\A$ is complex-linear, while in many of our examples $\A$ is conjugate-linear or there is no complex structure present.   Thus it is essential to study (\ref{m1}) as a {\em real} operator.

 \medskip

This paper has six sections.  Section \ref{sec2} introduces  the concentration condition  \eqref{co} and describes some analytic consequences.    

Section \ref{sec3} presents some elementary examples,   including one that extends a theorem of Taubes.  These arise as real linear perturbations of reducible Clifford bundles.  Example~1 is  classical, but already illustrates the idea that concentration occurs when a complex operator is perturbed by a conjugate-linear operator $\A$.  In \cite{t1}, C.~Taubes used a concentrating family of operators to give an interesting new proof of the Riemann-Roch Theorem for line bundles on complex curves;  our Example~2 extends this to higher-rank bundles over curves.  Example~3 shows how Witten's Morse Theory fits into the general setup of Section~2.

Sections \ref{sec4}, \ref{sec5} and \ref{sec6} give examples where $D$ acts on ``spinor-form pairs'', meaning sections of a subbundle of $W\oplus \Lambda^*(T^*X)$, where $W$ is a bundle of spinors and $\Lambda^*T^*X$ is the bundle of forms. Concentration occur along the zero set of a spinor field. These examples are most natural in low dimensions, especially in dimension four.  In particular, Examples~5 and 6 are natural perturbations of the linearized Seiberg-Witten equations. Example~6 in Section~5 suggest that there is a localization theorem that generalizes  the index localization Theorem~ \ref{Igor} to cases where $Z_\A$ is a union of submanifolds, rather than isolated points.
\medskip

The author would like to thank T.H. Parker for his suggestions and encouragement.

\medskip
   
  \vspace{1cm}
     
\section{Concentration Principle for Dirac Operators } 
\label{sec2}
\bigskip

This section describes some very general conditions in which one has a family $D_s$ of first order elliptic operators whose 
  low eigenvectors concentrate around submanifolds as $s\to \infty$.  
  
  \vspace{8mm}

Let $(X,g)$ be a closed Riemannian manifold  and $E,\;F$ be real vector bundles over $X$.  Suppose that
\begin{itemize}
\item $ D:\Gamma(E)\to \Gamma(F)$ is a first order elliptic differential operator with symbol $\sigma$, and
\item $\A:E\to F$  is  a real bundle map.
\end{itemize} 
From this data we can form the family of  operators
\begin{eqnarray*}
D_s=D+s\A
\end{eqnarray*}
where $ s\in\R$.  Furthermore, assuming that the bundles $E$ and $F$ have metrics, we can form the adjoint $A^*$, and the formal $L^2$ adjoint $D_s^*= D^*+ s\A^*$ of $D_s$.  With our conventions, the symbol of $D$ is the bundle map $\sigma :T^*X\to \Hom(E,F)$ defined by the equation
\[ 
D(f\xi) = \sigma(df)\xi + f D\xi, \qquad \forall f\in C^\infty(X),\ \xi\in \Gamma(E).  
\]
It follows that the symbol of $D^*$ is $-\sigma^*$. The main point of this paper is that such a family $D_s$ is especially interesting when $\A$ and the symbol $\sigma$ are related in the following way.

\begin{defn}[Concentrating pairs]
\label{MainDefCP}
In the above context, we say that  $(\sigma, \A)$  is a {\em concentrating pair} if it satisfies the algebraic condition
\begin{eqnarray}
\label{cond}
\A^*\circ \sigma(u) = \sigma(u)^*\circ \A,  \qquad \mbox{for every $u\in T^*X$.}
\end{eqnarray}
\end{defn}

\medskip

\begin{lemma}
\label{Weitzenboch}
A pair  $(\sigma, \A)$  is  a concentrating pair if and only if the operator 
$$
B_\A=D^{*}\circ \A+\A^{*}\circ D
$$
has order 0, that is, is a bundle map.   If so, then  for each $\xi\in C^{\infty}(E)$,
\begin{eqnarray}
\label{Dxiexpansion}
\|D_s\xi\|^2_2\, =\, \|D\xi\|^2_2\, +\, s^2\|\A\xi\|^2_2\, +\,  s \langle \xi,B_\A\xi\rangle
\end{eqnarray}
where these are $L^2$ norms and inner products, and hence
\begin{equation}
\label{2.ExpandD*D}
D_s^*D_s\ =\ D^*D +sB_\A +s^2\A^*\A
\end{equation}
\end{lemma}

\begin{proof}
Given a tangent vector $u\in T^*_pX$, choose a  smooth function $f$ with $df|_p=u$.  Then for any smooth section $\xi$ of $E$,
\begin{eqnarray*}
B_\A(f\xi)\, = \, D^{*}(f\A(\xi))\, +\, \A^{*}(D(f\xi)) & \, =\, & -\sigma^*(df)\A\xi + fD^*\A\xi +\A^*\sigma(df)\xi + f \A^*D\xi \\
&\, =\, &\big(-\sigma^*(u) \A+\A^*\sigma(u)\big)\xi  + f B_\A(\xi).
\end{eqnarray*}
Thus \eqref{cond} holds if and only if $B_\A(f\xi)=f B_\A(\xi)$, which means that $B_\A$ is a zeroth order operator.   To obtain \eqref{Dxiexpansion}, expand $|D+s\A|^2$ and integrate; this  gives 
$$
\|D_s \xi\|^2_2\ =\ \|D\xi\|^2_2\, +\, s^2\,\|\A(\xi)\|^2_2\, +\, s\,\langle D\xi,\A\xi\rangle\, +\, s\,\langle \A\xi, D\xi\rangle
$$
where, after integrating by parts,  the last two terms are equal to $s\,\langle \xi,B_\A\xi\rangle$.
\end{proof}

\bigskip

\begin{rem} 
Given $D$ and $A$ as above,  one can always form the self-adjoint operators
$$
{\cal D}\ =\ \begin{pmatrix}0&D^*\\D&0\end{pmatrix}
\qquad\mbox{and}\qquad
{\cal A}\ =\ \begin{pmatrix}0&A^*\\A&0\end{pmatrix}.
$$
Then $(\sigma_\D, \A)$ is a concentrating pair if and only if  \eqref{cond} holds i.e. $\sigma_{\cal D}(u)\circ{\cal A}=-{\cal A}\circ\sigma_{\cal D}(u)$ for all $u\in T^*X$; this occurs when both  $(\sigma, A)$ and $(-\sigma^*, A^*)$ are concentrating pairs. 
\end{rem}

\vspace{5 mm}

The assumption that $D:\Gamma(E)\to\Gamma(F)$ is elliptic  implies  that the bundles $E$ and $F$ have the same rank.   Thus a generic bundle map $\A:E\to F$ is an isomorphism at almost every point.  In the analysis of the family $D+s\A$, a key role will be played by the {\em singular set of $\A$}, defined as  the set
$$
Z_\A\ : =\ \big\{x\in X\, |\, \ker \A(x)\ne0 \big\}
$$
where  $\A$ fails to be injective.
\vspace{5 mm}

The following proposition shows the importance of the concentrating condition  \eqref{cond}.  It shows that, under Condition~\eqref{cond},  all solutions of $D_s\xi=0$ concentrate along the singular set $Z_\A$.   More generally, it shows that all solutions to the eigenvalue problem $D_s^*D_s\xi=\lambda(s) \xi$ with $\lambda (s)=O(s)$ also concentrate along $Z_\A$.

 For each $\delta>0$, let $Z(\delta)$ be the $\delta$-neighborhood of  $Z_\A$,  and let 
$$
\Omega(\delta)=X\setminus Z(\delta)
$$ 
be its complement.

\begin{prop}[Concentration Principle]
\label{concentrationProp}
 For each $\delta>0$ and $C\ge 0$, there is a constant  $C'=C'(\delta, \A, C)>0$, such that if $\xi\in C^{\infty}(E)$ has $L^2$ norm $\|\xi\|_2=1$ and satisfies $\|D_s\xi\|^2_2\leq C |s| $, then 
\begin{equation}
\label{concentrationestimate}
\int_{\Omega(\delta)} |\xi|^2\ dv_g\ <\frac{C'}{|s|}.
\end{equation}
\end{prop}

\begin{proof}
 Applying (\ref{Dxiexpansion}) to such a $\xi$ gives the inequalities
\begin{eqnarray*}
C |s|\,\geq\,\|D_s \xi\|^2_2\, \geq\, s\, \langle \xi,B_\A\xi\rangle\, +\,\ s^2\, \|\A(\xi)\|^2_2.
\end{eqnarray*}
By Lemma \ref{Weitzenboch},  $B_\A$ is a tensor  on the compact space $X$, so $M_1=  \sup_{X}|B_\A|$  is finite. Hence, by Cauchy-Schwartz,
$$
\left|\langle \xi,B_\A(\xi)\rangle\right|\ \le\ M_1\int_{X}|\xi|^2\ dv_g\ =\ M_1.
$$
 But $\A$ is injective on the fiber over each point $x\in X\setminus Z_\A$,  so there is a positive constant $\kappa(x)$ with $|\A_x(\xi)|\geq \kappa(x)|\xi|$.  By compactness,  there is a constant $\kappa>0$ with $\kappa(x)\ge \kappa$ on the closure of $\Omega(\delta)$ and therefore  
\begin{eqnarray*}
s^2\, \|A(\xi)\|_2^2\, \ge\, \kappa^2 s^2\int_{\Omega(\delta)}|\xi|^2\ dv_g.
\end{eqnarray*}
Combining these inequalities gives 
\begin{eqnarray*}
\int_{\Omega(\delta)}|\xi|^2\ dv_g\ \le\ \frac{M_1+  C}{\kappa^{2}|s|}.
\end{eqnarray*}
\end{proof}

\begin{rem}
\label{erates}
A bound of the form $|\A(\xi)|^2 \geq Cr^a |\xi|^2$ on a tubular neighborhood of $Z_\A$, where  $r$ is the distance  from $Z_\A$ and $a>0$, gives a bound on how the constant $C'$ in \eqref{concentrationestimate} depends on $\delta$. We will require such an assumption in (\ref{normalrates}) below.
\end{rem}

In the next corollary we obtain $C^{\ell,\a}$ estimates;   the first part of the proof was suggested by Akos Nagy.
\begin{cor}
Suppose $\xi\in C^{\infty}(E)$ is a section with $L^2$ norm 1 satisfying $D^*_sD_s\xi= \lambda_s \xi$ where $|\lambda_s|\leq C |s|$. For each $k\in \mathbb{N}$ and region $\Omega(\delta)$, there exist $C' = C'(\delta, k,\ell,\a, C)$ so that
$$
\|\xi\|_{C^{\ell,\a}(\Omega(\delta))} \leq C'|s|^{-\frac{k}{2}}.
$$
\end{cor}
\begin{proof}
First note that for every $\delta>0$ and integer $k\ge 1$,  there exist $C' = C'(\delta, k, C)$ so that whenever $\xi\in C^{\infty}(E)$  satisfies $\|\xi\|_2=1$ and  $\|D_s\xi\|^2_2\leq C |s| $, one has the bound  
\begin{equation}
\label{kconc}
\int_{\Omega(\delta)} |\xi|^2\ dv_g\ < C'|s|^{-k}  \qquad\mbox{for  $s\gg 0$.}
\end{equation}
This can proved by induction.  It holds for $k=1$ by (\ref{concentrationestimate}):  assume that it holds for $k\in \mathbb{N}$. Let $\rho$ a smooth cutoff function supported in $\Omega(\delta)$ with $\rho|_{\Omega(2\delta)}\equiv 1$ and note that
\begin{equation}
\label{2.2X}
D_s(\rho\xi) = \sigma(d\rho) \xi + \rho D_s\xi.     
\end{equation}
Integrating the squared norms over $\Omega(\delta)$ and  estimating from below using \eqref{2.ExpandD*D} yields
\begin{eqnarray}
\label{2.3X}
\int_{\Omega(\delta)}|D_s(\rho \xi)|^2dv_g &\geq& |s|\, \int_{\Omega(\delta)}\langle \xi,\rho^2 B_\A\xi\rangle dv_g\, +\,\ s^2\, \int_{\Omega(\delta)}|\rho\A(\xi)|^2dv_g\nonumber \\
 &\geq& -|s|M_1\int_{\Omega(\delta)}|\xi|^2dv_g \,+\, s^2 \kappa^2 \int_{\Omega(2\delta)}|\xi|^2dv_g, 
\end{eqnarray}
while integration by parts gives
$$
\int_{\Omega(\delta)}|\rho D_s\xi|^2dv_g\ =\  \int_{\Omega(\delta)}\langle \rho D_s\xi, \rho D_s\xi\rangle dv_g \ =\   \int_{\Omega(\delta)}\langle\xi, D_s^*(\rho^2 D_s\xi)\rangle dv_g 
$$
since $\rho$ vanishes on  $\partial\Omega(\delta)$.  But $ D_s^*(\rho^2 D_s\xi) = - 2\sigma^*(d\rho) (\rho D_s\xi) +\rho^2D^*_sD_s\xi$ with  $D^*_sD_s\xi= \lambda_s \xi$, so by Cauchy-Schwartz inequality followed by Young's inequality
\begin{eqnarray*}
\int_{\Omega(\delta)}|\rho D_s\xi|^2dv_g &=& -2\int_{\Omega(\delta)}\langle \sigma(d\rho)\xi, \rho D_s\xi\rangle dv_g + \lambda_s \int_{\Omega(\delta)} |\rho\xi|^2dv_g\\ &\leq& C_\delta\Big(\int_{\Omega(\delta)}|\rho D_s\xi|^2dv_g\Big)^{1/2}\Big(\int_{\Omega(\delta)}|\xi|^2 dv_g\Big)^{1/2} + \lambda_s\int_{\Omega(\delta)}|\xi|^2 dv_g  \\ 
 &\leq& \frac{1}{2}\int_{\Omega(\delta)}|\rho D_s\xi|^2dv_g\,+\,  (\frac{C^2_\delta}{2}+ \lambda_s)\int_{\Omega(\delta)}|\xi|^2dv_g. 
\end{eqnarray*}
Absorbing the common terms in the left hand side
\begin{eqnarray}
\label{2.4A}
\int_{\Omega(\delta)}|\rho D_s\xi|^2dv_g \leq (C^2_\delta+ 2\lambda_s) \int_{\Omega(\delta)}|\xi|^2dv_g.
\end{eqnarray}
Combining \eqref{2.2X}, \eqref{2.3X}, \eqref{2.4A} and using the inductive step then gives 
$$
s^2\kappa^2\int_{\Omega(2\delta)}|\xi|^2dv_g \ \leq\  \big(2\lambda_s+C_{\delta} + |s|M_1\big) \int_{\Omega(\delta)}|\xi|^2dv_g
 \ \leq \ \frac{C_{\delta}}{|s|^{k-1}}
$$
for an updated constant  $C_\delta$. Replacing $\delta$ by  $\frac{\delta}{2}$ finishes induction and establishes the bound \eqref{kconc}.

\medskip

Now  writing $D\xi = D_s\xi - s\A\xi$, we can apply  \eqref{2.4A} and \eqref{kconc} on $\Omega(2\delta)$ to obtain
$$
\|D\xi\|_{L^2(\Omega(2\delta))} \ \leq\  \|D_s\xi\|_{L^2\Omega(2\delta))}\,+\, s\|\A\xi\|_{L^2(\Omega(2\delta))}\, \leq \, C_{\delta,k} |s|^{-\frac{k}{2}}. 
$$
Combining this with the interior elliptic estimate for $D$ on  $\Omega(3\delta)$, we have
$$
\|\nabla\xi\|_{L^2(\Omega(3\delta))} \  \leq\     C_{\delta,k}\left(\|D(\xi)\|_{L^2(\Omega(2\delta))}  + \|\xi\|_{L^2(\Omega(2\delta))}\right)\ \le\  C'_{\delta,k} |s|^{-\frac{k}{2}}. 
$$
Replacing $\delta$ by  $\frac{\delta}{3}$  establishes the bound
\begin{eqnarray}
\label{2.6}
\|\xi\|_{L^{1,2}(\Omega(\delta))}\ \le\  C'_{\delta,k} |s|^{-\frac{k}{2}}
\end{eqnarray}
for $s\gg 0$.

Next, from \eqref{2.ExpandD*D} and the eigenvalue equation, we have 
\begin{equation}
\label{2.D*D+Ev}
D^*D\xi+sB_\A\xi+s^2\A^*\A \xi = D_s^*D_s\xi=\lambda_s\xi.  
\end{equation}
 The elliptic estimate for $D^*D$ then gives
$$
\|\xi\|_{L^{2,2}(\Omega(2\delta))}\ \leq\ C\big(\|D^*D\xi\|_{L^2(\Omega(\delta))}+ \|\xi\|_{L^2(\Omega(\delta))}\big) \   \leq  \  C' (1+s^2)\|\xi\|_{L^2(\Omega(\delta))}\ \le\     C''_\delta |s|^{-\frac{k-4}{2}},
$$
for large $s$.  Again,  we can  replace $\delta$ by $\frac{\delta}{2}$.

We can now bootstrap, repeatedly differentiating \eqref{2.D*D+Ev} and using elliptic estimates to obtain
$$
\|\xi\|_{L^{\ell,2}(\Omega(\delta))}\ \leq\ C_{\delta,k,\ell} |s|^{-\frac{k}{2}}
$$
for every $\ell$. Morrey's inequalities then give the stated bound on  the $C^{\ell,\a}$ norms. 
\end{proof}


The concentration condition fits nicely into  the context of Dirac operators. Recall that a vector space $V$  is a  representation of the Clifford algebra $C(\R^n)$ if there is a linear map $c:\R^n\to \End(V)$ that satisfies the Clifford relations
\begin{equation}
\label{0.Cliffordrelations}
c(u)c(v)+c(v)c(u)\ =\ -2 \langle u, v\rangle\ Id.
\end{equation}
for all $u, v\in \R^n$.  
  
\begin{lemma}
\label{lr}
The concentration condition \eqref{cond}  with $\sigma = c$ is equivalent to  
\begin{equation}
\label{condversion2}
 c(u)\circ \A^* = \A\circ c(u)^*\  \quad  \forall u\in T^*X.
\end{equation}
Hence the concentration principle applies to $D + s\A$ concentrates if and only if it applies for the adjoint operator  $D^* + s\A^*$.
\end{lemma}
\begin{proof}
Multiplying \eqref{cond} on  the left by $c(u)$ and on the right by $c(u)^*$ we get
$$
|u|^2 \A\circ c(u)^* = c(u)\circ \A^* |u|^2
$$ 
which gives \eqref{condversion2}.  The proof in the  opposite direction is similar.
\end{proof}

 We henceforth assume that    $E$ and $F$ are bundles  of equal rank, $E\oplus F$ admits a $\Z_2$ graded $\Spin^c$ structure  $c:\R^n\to \End(E\oplus F)$, and that $E\oplus F$ has  a $\Spin^c$ connection $\nabla$ that preserves the grading and satisfying $\nabla c = 0$. We also assume that $\A:E\to F$ is a bundle map that satisfes the concentrating condition \eqref{cond}. We can then form the Dirac operator 
\begin{equation}
\label{1.DefDirac}
D = c\circ \nabla : \Gamma(E)\rightarrow \Gamma(F)
\end{equation}
and the family $D_s=D+s\A$.  
\vspace{3mm}

We impose two further  conditions on $\A$ that  guarantee that the  components $Z_\ell$ of  the singular set  $Z_\A$ are submanifolds and that the rank of $\A$ is constant on each $Z_\ell$.  For this, we regard ${\cal A}$ as a section of a subbundle $\L$ of $\Hom(E,F)$ as in the following diagram:
\bear
\label{Ldiagram}
 \xymatrix{
\ \ \L\ \ \ar@{^{(}->}[r]  \ar[d]  & \Hom(E, F) \supseteq {\cal F}^l\\
(X, g) \ar@/^1pc/[u]^{\mathcal A} & 
}
\eear
Here $\L$ is a bundle that parameterizes some family of linear maps $A:E\to F$ that  satisfy  the concentration condition (\ref{cond})  for the operator \eqref{1.DefDirac}, that is, each  element $A\in\L$ satisfies  $A^*\circ c(u) = c(u)^*\circ A$  for every $u\in T^*X$.    Inside the total space of the bundle $\Hom(E, F)$,  the set of linear maps with $l$-dimensional kernel is a   submanifold 
 $\mathcal{F}^l$;    because $E$ and $F$ have the same rank, this submanifold  has codimension~$l^2$. In all of our examples $\L\cap {\cal F}^l$ is a manifold for every $l$. As a section of $\L$, $\A$ will be chosen transverse to  $\L\cap {\cal F}^l$ for every $l$. As a consequence of the Implicit Function Theorem, ${\cal A}^{-1}(\L\cap {\cal F}^l)$ will be a submanifold of $X$ for every $l$. The singular set decomposes as a union of these submanifolds and, even further, as a union of connected components $Z_\ell$:
\begin{equation}
\label{defZ_l}
Z_\A\, =\, \bigcup_l {\cal A}^{-1}(\L\cap {\cal F}^l)\, =\,\bigcup_\ell Z_\ell.
\end{equation}
$\A$ has constant rank along each $Z_\ell$, so $\ker \A$ and $\ker \A^*$ are well-defined bundles over $Z_\ell$. 
\vspace{10mm}

The special case when $Z_\A$ is a finite set of points was studied by Prokhorenkov and Richardson in \cite{pr}.   In this case,  an important role is played by certain  vector spaces $\K_p$ and $\hat{\K}_p$  associated with each $p\in Z_\A$.   These are defined as follows.

Fix  an isolated point $p \in Z_\A$. Let $K$ be the bundle obtained by parallel translating  $\ker \A_p$ along radial geodesics in a geodesic ball centered at $p$. We require that $\A$ has a non-degenerate zero at $p$ in the sense that 
\begin{eqnarray}
\label{normalrates}
\A^*\A|_K \ =\   r^2M + O(r^3)
\end{eqnarray}
where $r$ is the distance function from   $p$, and $M$ is a positive-definite symmetric endomorphism of  the bundle $K$. (It is shown  in \cite{pr} that $\A$ can always be perturbed to satisfy this condition.)
Choose an orthonormal frame $\{e_\a\}$ of the $T_pX$ with dual frame by $\{e^\a\}$ and define 
\begin{equation}
\label{IntroDefM}
M_\a = - c(e^\a)\nabla_{e_\alpha} \A_p:\,  \ker \A_p\rightarrow \ker \A_p \quad \mbox{and}\quad \hat{M}_\a = - c(e^\a) \nabla_\a \A^*|_p: \ker \A_p^* \to\ker \A_p^*.
\end{equation}
(Here, our sign convention  differs from that in \cite{pr}.)

Under assumption (\ref{normalrates}),  $\{M_\a\}$ and $\{\hat M_\a\}$ are two collections of commuting isomorphisms; each is self-adjoint by Condition \ref{cond}, and its spectrum is real,   symmetric and, by (\ref{normalrates},  does not contain 0. Consider the simultaneous positive eigenspace 
$$
K^+_\a = \bigoplus \Big\{\mbox{positive eigenspaces of}\ M_\a \Big\}.
$$  
\begin{defn}
\label{IntroDefKp}
For each  component  $p\in Z_\A$ define 
$$
\K_p = \bigcap_\a K_\a^+ \subset \ker \A_p.
$$
For the adjoint $\A^*$, $\hat\K_p\subset \ker \A^*_p$ is defined similarly  using the family $\{\hat M_\a\}$.
\end{defn}

The main result  in \cite{pr} is  the following localization theorem that expresses  the index of $D$ in terms of   the vector spaces   $\K_p$ and $\hat{\K}_p$.

\begin{theorem}
\label{Igor} 
Under the above assumptions the index of $D$ can be calculated by the local contributions from $Z_\A$ as
$$
\ind D = \sum_{p\in Z_\A} (\dim \K_p - \dim \hat\K_p).
$$
\end{theorem}

\vspace{10 mm}

\medskip
   
  \vspace{1cm}

\section{Basic Examples} 
\label{sec3}
\bigskip

The concentration condition (\ref{cond}) is an algebraic condition on the symbol $c$ of the Dirac operator $D$. Thus the search for concentrating pairs $(c, \A)$ is an algebraic problem about  representations of  Clifford algebras and their connection with geometry.  In  the next several sections,  we start with basic examples and progressively built more elaborate ones.

Our first two examples are in dimension two.  Both are perturbations of the form $D_s=D+s\A$ of a $\del$ operator by {\em conjugate-linear} zeroth-order operator $\A$.   Thus  $D_s$ is a {\em real} operator, although in the examples it is convenient to write $D_s$ using complex notation.

\vspace{5mm}
     
\noindent{\bf Example 1: } For smooth $L^2$ functions $f, g:\mathbb{C}\to\mathbb{C}$, consider the operators
$$
D_sf=\del f+ sz\bar{f}  \quad \mbox{and} \quad D'_s g = -\partial g+ sz\bar{g}.
$$ 
These have the form $D+s\A$ where  $\A$ is the self-adjoint real linear map  $\A f = z\bar{f}$.  Using Lemma~\ref{Weitzenboch}, the calculations
$$
{\mathcal B}_\A f\ =\ (\del^* \A + \A^* \del)f \ =\   - \partial (z\bar{f}) - z\overline{\del f}\ =\   - \bar{f} 
$$
and
\begin{equation}
\label{1.2.1}
{\mathcal B}'_\A f\ =\ (\del\A^* - \A \partial) f \ =\  \del (z\bar{f}) -z \overline{\partial f} \ =\  0  
\end{equation}
show that both  $(\sigma_{\bar\partial}, \A)$ and $(\sigma_{-\partial}, \A^*)$  are concentrating pairs. For these equations, we can find explicitly that $\ker D_s = \R e^{-s|z|^2}$ and $\ker D'_s = 0$.  The non-zero  solutions of $D_sf=0$ clearly concentrate at the origin as $s\to \infty$.

Similarly the equation $\del f + s\bar{z} \bar{f} = 0$ has only trivial solutions, and its adjoint has a one-dimensional real kernel. 

\vspace{5mm}

Next consider real Dirac operators on Riemann surfaces. In Section~7 of \cite{t1}, C. H. Taubes described a concentration property for perturbed $\del$-operators on complex line bundles over Riemann surfaces.   Example 2 generalizes Taubes observation to higher rank bundles:
\bigskip

\noindent{\bf Example 2:} Let $(\Sigma,g)$ be a closed Riemann surface with anticanonical bundle $\bar{K}$,  and let $E$ be a holomorphic bundle of rank $r$  with a Hermitian metric $\langle\cdot,\cdot\rangle$,    conjugate linear in the second argument.  The direct sum of the $\del$-operator $\del:\Gamma(E)\rightarrow \Gamma(\bar{K}E)$ and its adjoint is a 
self-adjoint Dirac operator 
$$
D= \begin{pmatrix}0&\bar\partial^*\\ \bar\partial& 0\end{pmatrix}:  \Gamma(E\oplus \bar{K}E) \to \Gamma(E\oplus \bar{K}E).
$$
The symbol of $D$, applied to a $(0, 1)$-form $u$ is
\begin{eqnarray}
\label{interior/exteriormult}
c(u)(\xi)\ =\ u\wedge\xi\ -\ \iota_u\xi,\quad \xi\in E\oplus\bar{K}E.
\end{eqnarray}
One checks that this satisfies the Clifford relations \eqref{0.Cliffordrelations}, so defines a Clifford bundle structure on $E\oplus\bar{K}E$.   Now choose
$$
\overline{\mu} \in\Gamma(\Sigma,\,\bar{K}\otimes_\mathbb{C} \Sym^2_\mathbb{C}E).
$$
Combined with the conjugate linear isomorphism $E\cong E^*$ defined by  the  hermitian metric,  $\overline{\mu}$ becomes a conjugate linear map $\mu : E\rightarrow \bar{K}E$.  Set 
$$
\A= \begin{pmatrix}0&\mu^*\\ \mu & 0 \end{pmatrix}\in\End_\R(E\oplus\bar{K}E).
$$

\vspace{3mm}

\begin{lemma}
\label{hrank}
$(c, \A)$ is a concentrating pair.
\end{lemma}
\begin{proof}
It suffices to fix a point $p\in \Sigma$ and verify that $c(u)\circ \A=- \A\circ c(u)$ for all $u\in T^*_p\Sigma$.  This is equivalent to proving that $\mu$ and its adjoint $\mu^*$ satisfy
the  two identities
\begin{eqnarray*}
\iota_u(\mu(\xi))\ =\ \mu^*(u\wedge\xi)\qquad\mbox{and}\qquad u\wedge\mu^*(\eta)\ =\ \mu(\iota_u(\eta))
\end{eqnarray*}
for all $\xi$ in the fiber $E_p$ and $\eta$ in $(\bar{K}\otimes E)_p$. Choose orthonormal bases $\{e_i\}$ of  $E_p$ and $\bar{k}$ of $\bar{K}$. Then $\overline{\mu} = \bar{k}\mu^{ij} e_i\otimes e_j \in \bar{K}\otimes_\mathbb{C} \Sym^2_\mathbb{C}(E)$ corresponds to   the map $\mu: E \rightarrow \bar{K}E$ defined by
$$
 \mu(\xi)\, =\, \bar{k}\langle e_i,\,\xi\rangle \mu^{ij}e_j.
$$
Thus for $u = \lambda \bar{k}$, we have
\begin{eqnarray*}
\iota_u \mu(\xi)\, =\, \bar\lambda\, \iota_{\bar{k}}( \bar{k} \mu^{ij}\langle e_i,\xi\rangle) e_i\,=\, \bar\lambda \mu^{ij}\langle e_i,\,\xi\rangle e_j
\end{eqnarray*}
and
$$
\mu^*(u\wedge \xi)\, =\, \langle \mu^*(u\wedge \xi),\, e_j\rangle e_j = \overline{\langle u\wedge\xi,\, \mu(e_j)\rangle} e_j = \bar\lambda\mu^{ji} \langle \bar{k} e_i,\, \bar{k}\xi\rangle e_j.
$$
These are equal since $\mu^{ij} = \mu^{ji}$. The second identity is proved from the first one using Lemma \ref{lr}.
\end{proof}

\vspace{5mm}

 Lemma~\ref{hrank} shows that  Proposition~\ref{concentrationProp} applies.   Thus   as $s\to\infty$ the low eigensections of the operator
$$
D_s\ =\ D+s\A: \Gamma(E\oplus \bar{K}E) \to \Gamma(E\oplus \bar{K}E)
$$
concentrate on the singular set $Z_\A$. The following lemma describes the structure of $Z_\A$.

\vspace{5mm}

\begin{lemma}
\label{crr}
For generic $\mu$, $Z_\A$ is a finite set of oriented points $\{p_\ell\}$.  Furthermore, 
\begin{itemize}
\item At each positive $p_\ell$,     $\K_\ell \cong \R$ and $\hat{\K}_\ell = 0$, and 
\item At each negative $p_\ell$,     $\K_\ell =0$ and $\hat{\K}_\ell \cong \R$.
 \end{itemize}
\end{lemma}
\begin{proof}
The singular set of $\A$ is the set of points in $\Sigma$ where $\mu : E\rightarrow \bar{K}E$ fails to be an isomorphism.  Thus $Z_\A$ is the zero set of $\det \mu:\Lambda^rE\to \Lambda^r(\bar{K}E)$.  Using the isomorphism $\Lambda^rE\cong \Lambda^rE^*$ of the induced hermitian metric on $\Lambda^rE$, this becomes a complex map 
$\Lambda^rE^*\to \Lambda^r(\bar{K}E)$, or equivalently a section  
$$
\det \mu \in \Gamma(L)
$$
of the complex line bundle
\begin{equation}
\label{1.LonSigma}
L= \bar{K}^r\otimes_\cx \Lambda^rE \otimes_\cx \Lambda^rE.
\end{equation}
Note that while $L$ is a holomorphic bundle, this section is only assumed to be smooth.  For a generic choice of $\mu$, the section $\det \mu$ will have only transverse zeros, which are therefore isolated points.  By compactness the set $\{p_\ell\}$ of zeros is finite.  At each $p_\ell$, the derivative $(\nabla\det\mu)$ is an isomorphism from $T_{p_\ell}\Sigma$ to the fiber of $L$ at $p$.  Both of these spaces are oriented;  $p_\ell$ is called positive if this isomorphism is orientation-preserving and is called negative if orientations are reversed.

 Let $z$ be a local holomorphic coordinate on $\Sigma$ centered at  $p\in\{p_\ell\}$.  Because $\det \mu$ has a zero at $p$, there is a non-vanishing section $e_1$ of $E$ so that $\mu(e_1)$ vanishes at $z=0$.  Since $\mu$ is conjugate-linear, the section $e_2=ie_1$ also satisfies $\mu (e_2) =0$ at $z=0$.   Hence we can choose  real local framings  of $E$ and  $\bar{K}E$ in which   $\mu$ has the local expansion
\begin{eqnarray*}
\mu\ =\ \begin{pmatrix}H &0\\ 0& * \end{pmatrix}\ +\  O(|z|^2)
\end{eqnarray*}
where $*$ denotes an invertible $(n-2)\times (n-2)$ real matrix and
$$ 
H: \ker\mu_0\rightarrow \ker\mu_0^*.
$$
is the real $2\times 2$ matrix that corresponds to multiplication by  $f\mapsto (\alpha z +\beta \bar z)\bar{f}$ under the identification $\cx=\R^2$.

For a generic section we have $|\alpha|\not= |\beta|$.  It follows that $\det \mu$ has a positive zero at $p$ if $|\alpha|>|\beta|$, and a negative zero if $|\alpha|<|\beta|$.

  Suppose $|\a|<|\b|$. By changing coordinates if necessary,  we may assume that $\a = 0$ and $\b=1$. One then sees that $\A^*\A$ has the expansion \eqref{normalrates}, so all of the assumptions  of the Theorem \ref{Igor} hold.   Write  $z = x + i y$,  and use the basis $\{e_1,e_2=i e_1\}$ of $\ker\mu_0$ and $\{d\bar{z}e_1, d\bar{z} e_2\}$ of $\ker\mu^*_0$ to write $ f = (f_1, f_2)\in \ker\mu_0$. Then   
$$
H(x,y)\begin{pmatrix} f_1\\ - f_2\end{pmatrix}  \ =\  \big(x A_1 +y A_2 \big)\begin{pmatrix}f_1\\ f_2\end{pmatrix}, 
$$
where
$$
A_1 = \begin{pmatrix} 1&0\\ 0 &-1\end{pmatrix}\quad \mbox{and}\quad A_2 = \begin{pmatrix} 0&-1\\ -1 &0\end{pmatrix}. 
$$
With respect to these basis, one can calculate that the Clifford multiplication  \eqref{interior/exteriormult} is  given by 
$$
c(dx) = -\tfrac{\sqrt{2}}{2} \begin{pmatrix} 1&0\\0&1\end{pmatrix} \quad \mbox{and}\quad c(dy) =  \tfrac{\sqrt{2}}{2}  \begin{pmatrix} 0&-1\\1&0\end{pmatrix},
$$
where these are maps  $\ker \mu_0^*\rightarrow \ker\mu_0$.  The corresponding matrices \eqref{IntroDefM} are therefore
$$
M_1 = - c(dx)A_1 = \tfrac{\sqrt{2}}{2}\begin{pmatrix}1&0\\0&-1 \end{pmatrix}\quad \mbox{and}\quad  M_2 = - c(dy) A_2 = \tfrac{\sqrt{2}}{2}\begin{pmatrix}-1&0\\0&1  \end{pmatrix}.
$$
Applying Definition~\ref{IntroDefKp}, one then sees that $\K_p = 0$ in this case. Analoguous calculations show that 
$$ 
\hat{M}_ 1= \tfrac{\sqrt{2}}{2}\begin{pmatrix}-1&0\\0&1 \end{pmatrix}\quad \mbox{and}\quad \hat{M}_2 = \tfrac{\sqrt{2}}{2}\begin{pmatrix}-1&0\\0&1  \end{pmatrix}, 
$$
and hence $\hat{\K}_p$ is one dimensional. 

The case $|\a|>|\b|$  is similar.

\end{proof}

\begin{cor}(Riemann-Roch)
\label{1.R-R}
If $E$ is a rank $r$ holomorphic bundle over a complex curve $C$, then
\begin{equation}
\label{1.RRLemma}
\ind \bar{\partial}_E \ =\ c_1(L)[\Sigma]\ =\  2c_1(E)[\Sigma] - r\chi(\Sigma).
\end{equation}
\end{cor}

\begin{proof}
Lemma~\ref{hrank} and the proof of Lemma~\ref{crr} show that the assumptions  of Theorem~\ref{Igor} hold.  In this case,  $Z=\{p_\ell \}$ is the set of zeros of a generic section $\det \mu$ of  the complex line bundle $L$ defined by \eqref{1.LonSigma}. By Lemma~\ref{crr}, each positive zero  has local contribution $\dim \K_\ell -\dim \hat{\K}_\ell = 1$, and similarly each negative  zero contributes $ -1$. Theorem \ref{Igor} therefore says that $\ind D$ is given by the Euler number
$$
\ind \bar{\partial}_E \ =\ \chi(L)[\Sigma]\ =\ c_1(L)[\Sigma].
$$
The  Riemann-Roch formula \eqref{1.RRLemma} follows because
$$
c_1(\Lambda^r E\otimes_\mathbb{C}\Lambda^r(\bar{K}\otimes_\mathbb{C} E))  \ =\  2c_1(\Lambda^r E) -  c_1(K^r) \ =\  2c_1(E)- rc_1(K)
$$
and $c_1(K)[\Sigma]= \chi(\Sigma)$.
\end{proof}

\vspace{5mm}

\noindent{\bf Example 3: } On a closed Riemmanian manifold $(X,g)$ the bundle $E\oplus F = {\Lambda}^\ev T^*X \oplus {\Lambda}^{odd}T^*X$ is a Clifford algebra bundle in two ways:  
\begin{eqnarray}
\label{hatcl}
\hat{c}(v) = v\wedge \oplus \ - \iota_{v^\#}\qquad \mbox{and}\qquad \tilde{c}(w) = w\wedge  \oplus \ \iota_{w^\#}
\end{eqnarray}
for $v,w\in T^*X$. One checks that 
\[
\hat{c}(v)^2 = - |v|^2 Id_{E\oplus F}\qquad \mbox{and}\qquad   \tilde{c}(w)^2 = |w|^2 Id_{E\oplus F}
\]
and that the adjoint operators with respect to the metric $g$ are   
\begin{eqnarray}
\label{eq:adjoints}
\hat{c}(v)^* = -\hat{c}(v) \qquad \mbox{and}\qquad \tilde{c}(w)^* = \tilde{c}(w), 
\end{eqnarray}
for every $v,w\in T^*X$. Finally $\hat{c}$ and $\tilde{c}$ anti-commute: 
\begin{eqnarray}
\label{twocliff}
\hat{c}(v)\tilde{c}(w)\ =\   - \tilde{c}(w) \hat{c}(v), \qquad \forall v,w\in T^*X.
\end{eqnarray} 
Note that $D=d+d^*$ is a  first-order operator whose symbol is $\hat{c}$. Fix a 1-form $\gamma$ with transverse zeros and set $\A_\gamma = \tilde{c}(\gamma)$.  Then, given relations \eqref{eq:adjoints}, equation \eqref{twocliff} translates to the concentration condition \eqref{MainDefCP} and Proposition~\ref{concentrationProp} shows that if
$$
D_s = D + s\A_\gamma = (d + d^*) + s\tilde{c}(\gamma) : \Omega^{ev}(X)\rightarrow \Omega^{odd}(X),
$$
then all solutions to the eigenvalue problem $D_s^*D_s\xi=\lambda(s) \xi$ with $\lambda (s)=O(s)$ concentrate around the zeros of $\gamma$.

This is the localization in  E.~Witten's well-known paper on Morse Theory \cite{w1}.    
\medskip
   
  \vspace{1cm}

\section{Clifford Pairs }
\label{sec4}
\bigskip

Examples~1-3 can be extended and placed in a general context by working with  Clifford algebra bundles.   A bundle $W\to X$ is called a {\em Clifford algebra bundle} if it is equipped  with a bundle map 
$$
c:Cl(T^*X)\to \End(W)
$$ 
that is an algebra homomorphism, meaning that it satisfies the Clifford relation \eqref{0.Cliffordrelations}.  For each  connection $\nabla$ on $W$ satisfying $\nabla c=0$, there is an associated Dirac operator $D=c\circ \nabla$ on $\Gamma(W)$ whose symbol is $c$.  This section shows how interesting examples arise by taking $W$ to be the direct sum of two   Clifford
bundles associated with different representations of the groups $\Spin(n)$ or $\Spin^c(n)$ (cf. \cite{lm}).

To describe the general context, let  $(E, c)$ and $(E',\, c')$ be  Clifford algebra bundles on $(X,g)$ with connections and with corresponding Dirac operators $D$ and $D'$. Suppose there is a bundle map ${\cal P}: E'\rightarrow E$;  one can then consider the diagram
\begin{equation}
\label{1.3.diagram1}
\setlength{\unitlength}{1mm}
\begin{minipage}{3cm}

\begin{picture}(55,25)
\put(12,22){\(E\)} \put(40,22){\(E\)}\put(22,21.5){\vector(1,0){14}}
\put(15,9){\vector(0,1){10}}
\put(12,3){\(E'\)}
\put(8,13){\({\cal P}\)}
\put(25,23){\(c(v)\)}
\put(41,9){\vector(0,1){10}}\put(40,3){\(E'\)}\put(43,13){\({\cal P}\)}
\put(22,5){\vector(1,0){14}}\put(25,7){\(c'(v)\)}
\end{picture}
\end{minipage}
\end{equation}
for each $v\in T^*X$.
The perturbed operator
$$
{\cal D}_s= {\cal D} + s\,{\cal A}:\Gamma(E\oplus E')\rightarrow\Gamma(E\oplus E')
$$
with
\begin{eqnarray*}
{\cal D}\ =\ \begin{pmatrix}D&0\\0&D'\end{pmatrix}
\qquad\mbox{and}\qquad
{\cal A}\ = \, \begin{pmatrix}0&{\cal P}\\- {\cal P}^*&0\end{pmatrix}
\end{eqnarray*}
then satisfies the concentration principle if and only if  Diagram~\eqref{1.3.diagram1} commutes for every  $v\in T^*X$.  Note that  if $E$ and $E'$ are reducible Clifford bundles,  one can restrict ${\cal D}_s$ to sub-bundles to produce additional examples of concentrating pairs.

The examples in this section are special cases in which we take $E$ and $E'$ to be of the form $W\otimes \Lambda^*(T^*X)$ where $W$ is a bundle of spinors. We next describe this setup, beginning with some linear algebra.

\vspace{5 mm}

Let $\Delta$ be the fundamental $\Spin^c$ representation of the group $\Spin^c(n)$;  $\Delta$ is irreducible for $n$ odd and the sum $\Delta^+\oplus \Delta^-$ of two irreducible representations for $n$ even (see \cite{lm}). Clifford multiplication is a linear map $c :\R^n\to\End_\cx(\Delta)$;  we will often use Hitchin's ``lower dot'' notation
$$
v.\phi:= c(v)(\phi).
$$
There is also a Clifford algebra map  $\hat{c}: \R^n\to\End\left(\Lambda^*\R^n\right)$  given by
$$
\hat{c}(v):=\sigma_{d + d^*}(v) =(v\wedge\cdot) -\iota_{v}(\cdot).
$$
 These two Clifford multiplications are intertwined in the following sense.

\begin{lemma}
\label{3.Lemma1}
Clifford multiplication extends to a $\Spin^c(n)$-equivariant linear map $c:\Lambda^*\R^n\to \End_\cx(\Delta)$ that satisfies
\begin{eqnarray}
\label{clhat}
v. b. \psi =  (\hat{c}(v)b).\psi  \qquad \mbox{for all $v\in \R^n,\ b\in \Lambda^* \R^n$ and $\psi\in\Delta$}.
\end{eqnarray}
\end{lemma}

\begin{proof}
Define the extension $c:\Lambda^*\R^n\to \End_\cx(\Delta)$ using the standard basis $\{e^j\}$  of $\R^n$  by
\begin{eqnarray}
\label{formact}
c(e^1\wedge\dots\wedge e^p)\ =\   e^1_\cdot\dots e^p_\cdot
\end{eqnarray}
for each  $p$-tuple $(i_1,\dots, i_p)$ with $i_1<\dots <i_p$. This map is Spin($n$)-equivariant because  for every $g\in \Spin($n$)$ and $\eta \in {\Lambda}^*X$ we have that
$$
c(\Ad(g)^*\eta) = g_\cdot c(\eta) g^{-1}_\cdot
$$
Indeed, if $\eta = e^1\wedge\dots\wedge e^p$ then,  since $\{\Ad(g)^*e^i\}$ is also an orthonormal coframe with the same orientation,   \eqref{formact} implies that
\begin{eqnarray*}
c(\Ad(g)^* \eta)(g_\cdot\psi)\, &=&\,   c(\Ad(g)^* e^1)c(\Ad(g)^*e^2)\dots c(\Ad(g)^*e^p)(g_\cdot\psi)\\[2mm]
 &=&   c(\Ad(g)^* e^1)c(\Ad(g)^*e^2)\dots (g_\cdot(c(e^p)\psi)\\ 
 &=& g_\cdot \left[( c(e^1)c(e^2)\dots c(e^p)\right] g.^{-1}(g\psi)\\[2mm]
 &=& g_\cdot( c(\eta)\psi).
\end{eqnarray*}
To verify \eqref{clhat} note that for all $k,\,l$
\begin{eqnarray*}
e^l_\cdot (e^1\wedge\dots\wedge e^k)_\cdot\psi\, &=&\, e^l_\cdot e^1_\cdot\dots e^k_\cdot \psi \\[2mm]
&=& \begin{cases} 
(e^l\wedge e^1\wedge\dots\wedge e^k)_\cdot \psi , & \text{if}\  l>k  \\[2mm]
  (-1)^l(e^1\wedge\dots\wedge\hat{e^l}\wedge\dots\wedge e^k)_\cdot\psi\, & \text{if}\ 1\le l\le k
\end{cases}\\[2mm]
&=& \begin{cases}  (e^l\wedge e^1\wedge\dots\wedge e^k)_\cdot \psi, \hspace{16mm}& \text{if}\  l>k  \\[2mm]
 - (\iota_{e^l}(e^1\wedge\dots\wedge e^k))_\cdot\psi, & \text{if}\ 1\le l\le k  
\end{cases}\\[2mm]
&=& \hat{c}(e^l)(e^1\wedge\dots\wedge e^k)_\cdot\psi
\end{eqnarray*}
\end{proof}

\vspace{5mm}

Because of $\Spin^c(n)$-equivariance,  the map of Lemma~\ref{3.Lemma1} globalizes. Let $(X, g)$ be an oriented Riemannian $n$-manifold with a $\Spin^c$  bundle $W$,  a Hermitian metric $\langle\cdot,\cdot\rangle$ on $W$,  and determinant bundle $L =\det_{\mathbb{C}}(W)$. Clifford multiplication defines  bundle maps

\begin{eqnarray}
\label{3}
c :\Lambda^*T^*X\rightarrow \End_{\mathbb{C}}(W)  \qquad \mbox{and}\qquad  \hat{c}: T^*X\to\End\left(\Lambda^*T^*X\right)
\end{eqnarray}
that satisfy \eqref{clhat}. Given a Hermitian connection $A$ on $L$ with curvature $F_A$,  we get an induced spin covariant derivative $\nabla^A$ on $W$ compatible with the Levi-Civita connection $\nabla$ on $T^*X$ and a Dirac operator $D_A$ on $W$. 

\vspace{5mm}

\noindent{\bf Example 4: \it Spinor-form pairs} 

In the above context, consider the map  
\begin{equation}
\label{firstPpsi}
P: W \rightarrow {\Hom}_{\mathbb{C}}\Big({\Lambda}^*T_{\mathbb{C}}^*X,\,W\Big) \qquad \mbox{by}\ \psi\mapsto c(\cdot)\psi.
\end{equation}
For each spinor field $\psi\in\Gamma(W)$ we consider the operator
$$
{\cal D}_s= {\cal D} + s\,{\cal A}_\psi:\Gamma(W\oplus{\Lambda}^*_\mathbb{C} T^*X)\rightarrow\Gamma(W\oplus{\Lambda}^*_\mathbb{C}T^*X)
$$
with
\begin{eqnarray*}
{\cal D}\ =\ \begin{pmatrix}D_A&0\\0&d+ d^*\end{pmatrix}
\qquad\mbox{and}\qquad
{\cal A}_\psi\ = \,\begin{pmatrix} 0&P_\psi\\- {P_\psi}^*&0\end{pmatrix}
\end{eqnarray*}
where  $ {P_\psi}^*$  denotes the complex adjoint of $ P_\psi$.

\vspace{3mm}

\begin{lemma}
\label{1.3lemma}
 $(\sigma_{\cal D} , \A_\psi)$ is a concentrating pair. 
\end{lemma}
\begin{proof}
The symbol of ${\cal D}$, applied to a covector $v$, and its adjoint are given by
\begin{eqnarray*}
{\sigma_{\cal D}(v)}\ =\ \begin{pmatrix}c(v)&0\\0& \hat{c}(v)\end{pmatrix} \qquad\mbox{and}\qquad { \sigma_{\cal D}(v)^*}\ =\ \begin{pmatrix}-c(v)(\cdot)&0 \\0& -\hat{c}(v)\end{pmatrix}
\end{eqnarray*}

Formula $(\ref{clhat})$ expresses the fact that  the diagram
\begin{eqnarray*}
\label{3.WLdiagram}
\setlength{\unitlength}{1mm}
\begin{picture}(55,25)
\put(12,22){\(W\)} \put(40,22){\(W\)}\put(22,21.5){\vector(1,0){14}}
\put(15,9){\vector(0,1){10}}
\put(12,3){\({\Lambda}^*X\)}
\put(8,13){\(P_\psi\)}
\put(25,23){\(c(v)\)}
\put(41,9){\vector(0,1){10}}\put(40,3){\({\Lambda}^*X\)}\put(43,13){\(P_\psi\)}
\put(22,5){\vector(1,0){14}}\put(25,7){\(\hat{c}(v)\)}
\end{picture}
\end{eqnarray*}
commutes for every $v\in T^*X$, which means that   $(\sigma_{\cal D} , \A_\psi)$ is a concentrating pair. 
\end{proof}

\vspace{10mm}

Unfortunately, Lemma~\ref{1.3lemma} does not automatically mean that the concentration theorems of Section~2  apply to general spinor-form pairs.  The difficulty is seen when one examines the singular set
$$
Z_\A\ =\ \{x\in X\, |\, \ker P_\psi \not= 0\}.
$$
The dimension of the exterior algebra $\Lambda^*(\R^n)$ is $2^n$, and the fundamental representation of $Spin(n)$ has complex dimension $2^{[\frac{n}{2}]}$. It follows (see the chart) that  whenever $\dim X>2$, every linear  map $P_\psi: \Lambda^*(T^*X)\to W$ has a non-trivial kernel at each point, so $Z_\A$ is all of $X$.

\vspace{2mm}
\begin{table}[h]
\centering
\setlength\tabcolsep{3mm}
\begin{tabular}[h]{ccccccc}
$n$ & 2 & 3 & 4 & 5 & 6 & 7\\[2mm]
$\dim_\R \Lambda^*(\R^n)\ \ \ \ $ & 4 & 8 & 16 & 32 & 64 &128 \\[2mm]
$\dim_\R W$ & 4 & 4 & 8 & 8 & 16 & 16\\[2mm]
\end{tabular}
\end{table}
\vspace{2mm}
 
 To avoid this difficulty, we look for  sub-bundles $\L$  of $\Hom(\Lambda^*(T^*X), W)$  as in  diagram~\eqref{Ldiagram}.  One way to obtain such sub-bundles is via bundle involutions.

\bigskip

Suppose that $T = \begin{pmatrix}\tau&0\\ 0&\hat{\tau}\end{pmatrix}$ is a metric-preserving bundle involution of $E\oplus F$  that satisfies
\bear
\label{1.EF}
\nabla T = 0  \quad\mbox{and}\quad \sigma_D(v)\tau\ =\ \pm\,\hat{\tau}\sigma_{D}(v) 
\eear
for every covector $v$. Let $E=E^+\oplus E^-$ and $F=F^+\oplus F^-$ be the decompositions into $\pm 1$ eigenspaces of $\tau$ and $\hat{\tau}$ with corresponding projections $p^\pm = \tfrac{1}{2}(1_E\pm \tau) : E\rightarrow E^\pm$ and  $\hat{p}^\pm = \tfrac{1}{2}(1_F \pm \hat{\tau}) : F\rightarrow F^\pm$. Set 
$$ 
D^+ = \sigma_D\circ p^+\nabla|_{E^+}\qquad \mbox{and}\qquad  A^+ = \hat{p}^\pm A|_{E^+}
$$ 
where the sign on $ \hat{p}^\pm$ is fixed to be the sign in \eqref{1.EF}.  Then 
 $D^+$ and $A^+$ are linear maps from $\Gamma(E^+)$  to $\Gamma(F^\pm)$,  again with the sign being the sign in   \eqref{1.EF}.

\begin{lemma}
\label{InvolutionProp}
If $(\sigma_D, A)$ satisfies the concentration condition (\ref{cond}), then so does $(\sigma_{D^+}, A^+)$.  Thus for
\bear
\label{1.D^+}
D^+_s = D^++s A^+:\Gamma(E^+)\to \Gamma(F^\pm),
\eear
all solutions to the eigenvalue problem $D_s^{+*}D^+_s\xi=\lambda(s) \xi$ with $\lambda (s)=O(s)$ concentrate along $Z_{A^+}$.
\end{lemma}

\begin{proof}
The operator $ p^+\nabla|_{E^+}$ defines a metric compatible connection on sections of $E^+\rightarrow X$. Also
$$
(A^+)^*\sigma_D(v)|_{E^+} + (\sigma_D(v)|_{E^+})^*A^+ = p^+(A^*\sigma_D(v) + \sigma^*_D(v)A)|_{E^+} = 0
$$
for every $v\in T^*X$.
\end{proof}

In the examples below, we will build involutions by combining three bundle maps.  All three are defined when  $X$ is an  oriented Riemannian $n$-manifold.
 \begin{itemize}
\item The parity operator $(-1)^p$ that is  $(-1)^p Id$ on $p$-forms. 
\item The Hodge star operator, which satisfies $*^2= (-1)^{p(n-p)}$.
\item Clifford multiplication by the volume form $d\text{vol}$, which satisfies $(d\text{vol}).^2= (-1)^{[n/2]}$.
\end{itemize}
 These can be used to define  two involutions on spinor-form pairs:

\medskip

\noindent{\bf The parity involution.}  When $\dim X=2n$ is even,  the parity operator
$$
\tau = \hat{\tau} = i^{n} d\text{vol}_{\cdot}\oplus (-1)^{p+1} \in \End(W\oplus {\Lambda}^p_\mathbb{C} X)
$$
is  an involution; its $\pm 1$ eigenbundles are 
$$
E = W^+\oplus{\Lambda}^{odd}_\mathbb{C} X \qquad \text{and} \qquad  F = W^-\oplus{\Lambda}^{ev}_\mathbb{C} X, 
$$
and $\sigma_{\cal D}(v) T = - T \sigma_{\cal D}(v)$.  Furthermore, the restriction of \eqref{firstPpsi} decomposes as 
$$
W^+ \rightarrow {\Hom}_{\mathbb{C}}\Big({\Lambda}^{ev}T_{\mathbb{C}}^*X,\,W^+\Big)\oplus {\Hom}_{\mathbb{C}}\Big({\Lambda}^{odd}T_{\mathbb{C}}^*X,\,W^-\Big).
$$
Thus for  any $\psi\in \Gamma(W^+)$, we can write  $P_\psi = P^{ev}_\psi+P^{odd}_\psi$ under this decomposition, and set 
\begin{eqnarray}
\label{evodd}
{\cal A}^+_\psi\ = \, \begin{pmatrix}0&P^{odd}_\psi\\- P^{ev*}_\psi &0\end{pmatrix}.
\end{eqnarray}
Then by Lemma~\ref{InvolutionProp} the operator
$$
{\cal D}^+_s= {\cal D}^+ + s\,{\cal A}^+_\psi:\Gamma(W^+\oplus{\Lambda}^{odd}_\mathbb{C} X)\rightarrow\Gamma(W^-\oplus{\Lambda}^{ev}_\mathbb{C} X)
$$
satisfies the concentration condition \eqref{cond}.

\medskip

\noindent{\bf The self-duality involution.}  When $\dim X=4n$ there is a second involution on the bundles $ W^{\pm}\oplus{\Lambda}^{odd/ev}_\mathbb{C} X$, namely the self-duality involution 
\[
T : W\oplus \Lambda^*_\mathbb{C}X \to W\oplus \Lambda^*_\mathbb{C}X, \quad T(\phi, \alpha) = \begin{cases} (\phi, \alpha), \quad &\text{if}\ \alpha\in \Lambda_\mathbb{C}^pX,\ 0\leq p <2n
\\
 (\phi, *\alpha), \quad &\text{if}\ \alpha\in \Lambda_\mathbb{C}^{2n}X
\\
  (\phi, -\alpha), \quad &\text{if}\ \alpha\in \Lambda_\mathbb{C}^pX,\ 2n< p \leq 4n
\end{cases}.
\] 
This commutes with the parity involution and $\sigma_{\cal D}(v) T = T \sigma_{\cal D}(v)$ and ${\cal A}^{+*}_\psi T = T{\cal A}^+_\psi$ for $\psi\in W^+$. Let $\A^{++}$ denote the restriction of $\A^+$ to the positive eigenspace of $T$.  Then 
\begin{equation}
\label{sellfdualityspinor/form}
{\cal D}^{++} + s{\cal A}_\psi^{++} : \Gamma(E)\rightarrow\Gamma(F)
\end{equation}
satisfies the concentration condition \eqref{cond},  where
\begin{eqnarray}
\label{half}
E = W^+\oplus \left(\sum_{p=1}^{n}{\Lambda}_\mathbb{C}^{2p-1} X\right)
\qquad \mbox{and}\qquad 
F = W^-\oplus \left({\Lambda}_\mathbb{C}^{2n, +} X  \oplus \sum_{p=0}^{n-1}{\Lambda}_\mathbb{C}^{2p}X\right).
\end{eqnarray} 
In the next section, we will consider the concentrating operator (\ref{sellfdualityspinor/form}) in dimension four.
\medskip
   
\vspace{1cm}

\setcounter{equation}{0}
\section{Self-dual spinor-form pairs in dimension 4 } 
\label{sec5}
\bigskip

When  $X$ is a closed oriented Riemannian 4-manifold,  the  self-duality involution produces a Dirac operator \eqref{sellfdualityspinor/form} with the concentration  property and with  a singular set $Z_\A$ that, we will show next, is not all of $X$.    

In dimension four, the self-dual spinor-form bundles, simply are
\begin{eqnarray}
\label{5.1A}
E = W^+\oplus {\Lambda}^1X \quad \mbox{and}\quad F = W^-\oplus ({\Lambda}^0\oplus {\Lambda}^{2,+}X).
\end{eqnarray}
The operator (\ref{sellfdualityspinor/form}) becomes a Dirac operator after a slight modification: dropping the ++ subscripts and inserting factors of $\sqrt{2}$, we consider ${\cal D}  +s \A$ for 
$$
\sigma_{\cal D}(v) = \begin{pmatrix} c(v)&0\\ 0& \hat{c}(v)\end{pmatrix} \quad \text{and} \quad {\cal A}_\psi\, = \, \begin{pmatrix}0& P^{odd}_\psi\\- {P^{ev}_\psi}^*&0\end{pmatrix}
$$ 
where
$\hat{c}(v)$ is the symbol map of the Dirac operator $\sqrt{2}d^++d^*$ and
$$
P_\psi^{odd} : {\Lambda}^1\,\rightarrow\, W^-, \quad b\mapsto  b_\cdot \psi\qquad \text{and}\qquad P_\psi^{ev}: {\Lambda}^0\oplus{\Lambda}_\mathbb{C}^{2, +}\,\rightarrow W^+\quad (\rho,\, \theta)\mapsto ( \rho + \tfrac{1}{\sqrt{2}} \theta)_\cdot\psi. 
$$
Then, the diagram
\begin{center}
\setlength{\unitlength}{1mm}
\begin{picture}(55,25)
\put(8,22){\(W^+\)} \put(40,22){\(W^-\)}\put(22,21.5){\vector(1,0){14}}
\put(11,9){\vector(0,1){10}}
\put(3,3){\({\Lambda}^0\oplus{\Lambda}^{2,+}\)}
\put(3,13){\(P_\psi^{ev}\)}
\put(25,23){\(c(v)\)}
\put(41,9){\vector(0,1){10}}\put(40,3){\({\Lambda}^1\)}\put(43,13){\(P_\psi^{odd}\)}
\put(24,5){\vector(1,0){12}}\put(25,7){\(\hat{c}(v)\)}
\end{picture}
\end{center}
commutes.

The following algebraic fact is unique to dimension 4, stemming from the isomorphism $\Spin(4) \simeq SU(2)\times SU(2)$. 

\begin{lemma}
\label{lemma5.1}
$W^+$ is a Clifford bundle for the bundle of Clifford algebras $Cl(\Lambda^{2,+}(X))$.
\end{lemma}
\begin{proof}
It suffices to show this at a point $p\in X$. Using an orthonormal coframe  $\{e^i\}$, 
$$
\eta_0\, =\, \tfrac{1}{\sqrt{2}}( e^1\wedge e^2\,+\, e^3\wedge e^4),\qquad \eta_1\, =\, \tfrac{1}{\sqrt{2}}( e^1\wedge e^3\,+\, e^4\wedge e^2),\qquad \eta_2\, =\, \tfrac{1}{\sqrt{2}}(e^1\wedge e^4\, +\,  e^2\wedge e^3)
$$
are an orthonormal basis of ${\Lambda}_p^{2,+}(X)$. Note that 
$$
\eta_{i\cdot}\eta_{j\cdot}\, +\,\eta_{j\cdot}\eta_{i\cdot}\, =\,  2\langle \eta_i,\,\eta_j\rangle(d\text{vol}_\cdot\, -\, Id_W)
$$
for every $i, j$.  Using linearity and restricting to $W^+$,  we obtain 
\begin{eqnarray}
\label{7}
\eta_\cdot\theta_\cdot\, +\,\theta_\cdot\eta_\cdot\, =\,  -4\langle \eta,\,\theta\rangle Id_{W^+}
\end{eqnarray}
for every $\eta,\, \theta\in {\Lambda}^{2,+}(X)$.  This is an analog of the Clifford relation for the self-dual 2-forms acting on $W^+$, and therefore  finishes the proof.
\end{proof}

\begin{rem}
\label{clbasis}
Fix $\psi\in W^+\backslash \{0\}$ and let $\eta_0, \eta_1, \eta_2$ as above. Then, by the  proof of Lemma~\ref{lemma5.1}, the set $\{\psi,\eta_{k\cdot}\psi\}\subset W^+$ is orthogonal and $\eta_{k\cdot}\eta_{k\cdot}\psi = -2\psi$ which implies $|\eta_{k\cdot}\psi|^2 = 2|\psi|^2$. 
Both $E$ and $F$ are 8-dimensional real vector bundles. The volume form acts on $E$ by  $d\text{vol}.\SM{\phi\\b}= \SM{-\phi\\ \ b}$.   Setting $\xi =\SM{\phi\\ b}$ with    $|\phi|=|b|=\frac{1}{\sqrt{2}}$, the Clifford action produces orthonormal bases for $E_p$ and $F_p$ so that 
$$
E_p\, =\,\span\{e^I_\cdot  \xi: I\ \text{length even}\,\} \quad \mbox{and}\quad  F_p\,=\,\span\{e^J_\cdot \xi: J\ \text{length odd}\,\}
$$
where $e^I.\xi$ denotes $e^{i_1}. e^{i_2}.\dots e^{i_\ell}.\xi$ for each string  $I=(i_1, \dots, i_\ell)$ with $i_1<i_2<\cdots <i_\ell$.
\end{rem}

\vspace{5mm}

Regard now $W^+$ as a real vector bundle of rank 4 with the induced metric. By considering the negative definite quadratic form produced by that metric we can form the algebra bundle $Cl^{0,4}(W^+)$. The perturbation $\A_\psi$ enjoys the following property: 

\begin{lemma}
\label{clifford}
The map $W^+\rightarrow \End(E\oplus F)$ given by $\psi\mapsto \begin{pmatrix}0&{\cal A}_\psi^*\\{\cal A}_\psi&0\end{pmatrix}$ defines a   representation of the real Clifford algebra bundle $Cl^{0,4}(W^+)$ on $E\oplus F$.
\end{lemma}
\begin{proof}
Fix $\psi \in W_p^+\setminus 0$. By the Clifford relations, the sets $\{e^k_\cdot\psi\}\subset W^-$ and $\{\psi,\eta_{k\cdot}\psi\}\subset W^+$ are orthogonal. Therefore for $b=\sum b_l e^l\in \Lambda^1_\mathbb{C}$,
$$
{ P^{odd}_\psi}^*\circ  P^{odd}_\psi(b)\, = \,b_l{ P^{odd}_\psi}^*\circ  P^{odd}_\psi(e^l) =\, \Re\langle e^l_\cdot\psi,\,e^k_\cdot\psi\rangle b_l e^k\,=\,|\psi|^2 b, 
$$ 
and similarly for $(\rho,\theta)\in {\Lambda}_\mathbb{C}^0\,\oplus\,{\Lambda}_\mathbb{C}^{2,+}$
\begin{eqnarray*}
{ P^{ev}_\psi}^*\circ  P^{ev}_\psi(\rho, \theta)\, &=& \rho |\psi|^2\, +\, \frac{1}{\sqrt{2}}\theta_l\langle \eta_{l\cdot}\psi, \psi \rangle + \frac{1}{\sqrt{2}} \rho\langle\psi,\eta_{k\cdot}\psi \rangle \eta_k +\frac{1}{2}\theta_l \langle\eta_{l\cdot} \psi,\eta_{k\cdot}\psi \rangle \eta_k \\ &=& \rho |\psi|^2\, +\, \frac{1}{2}\langle \eta_{k\cdot}\psi, \eta_{k\cdot}\psi \rangle \theta_k\eta_k \\ &=&\, |\psi|^2(\rho, \theta).
\end{eqnarray*}
This proves that 
\begin{eqnarray*}
{\cal A}_\psi^*\circ{\cal A}_\psi\, =\, |\psi|^2 Id_{E} \qquad \text{and}\qquad {\cal A}_\psi\circ{\cal A}^*_\psi\, =\, |\psi|^2 Id_{F}.
\end{eqnarray*}
Finally,  polarization gives the relations
\begin{eqnarray*}
{\cal A}_{\psi_1}^*\circ{\cal A}_{\psi_2}\, +\, {\cal A}^*_{\psi_2}\circ{\cal A}_{\psi_1}\, =\, 2\Re\langle\psi_1,\,\psi_2\rangle Id_{E}
\end{eqnarray*}
and
\begin{eqnarray*}
{\cal A}_{\psi_1}\circ{\cal A}^*_{\psi_2}\, +\, {\cal A}_{\psi_2}\circ{\cal A}^*_{\psi_1}\, =\, 2\Re\langle\psi_1,\,\psi_2\rangle Id_{F}
\end{eqnarray*}
for every $\psi_1, \psi_2\in W^+$. Hence we get a well-defined {\em algebra} map 
$$
Cl^{0,4}(W^+)  \rightarrow \End(E\oplus F) \qquad\mbox{by }  \psi\mapsto \begin{pmatrix}0&{\cal A}_\psi^*\\{\cal A}_\psi&0\end{pmatrix}.
$$
The lemma follows.
\end{proof}
\begin{cor}
\label{4}
 The mapping $\psi\mapsto {\cal A}_\psi$ defines an injection $W^+ \to \Isom(E, F)$.
\end{cor}
\begin{proof} 
Let $\xi\in E$. Then by Lemma \ref{clifford} 
$$
|{\cal A}_\psi \xi|^2\, =\, |\xi|^2|\psi|^2,
$$
implying that if $\xi\in \ker{\cal A}_\psi$ is nontrivial then $\psi = 0$. Therefore $\psi \neq 0$ if and only if ${\cal A}_\psi$ is non-singular,  which implies the corollary.
\end{proof}

\vspace{6mm}

\noindent{\bf Example 5:}
Fix a section $\psi\in \Gamma(W^+)$ that is transverse to zero, and consider the  operator $$\D + s\A_\psi:\Gamma(E)\to\Gamma(F)$$ for the bundles \eqref{5.1A}.  Because $W^+$ has rank 4 as a real vector bundle, Corollary~\ref{4} implies that  the singular set $Z_\psi$ of $\A_\psi$ will be a finite set of oriented points.
Fix a point $p\in Z_\psi$,  a coordinate chart $(U, \{x_\a\})$ with $\,x_\a(p) = 0$, and tangent frame  $\{e_\a\}$ with dual coframe $\{e^\a\}$ respectively. In this basis, we can write
$$
\psi(x) = x_\a\psi_\a + O(|x|^2)
$$ 
for some elements $ \psi_\a\in W_p^+$.  Extend these smoothly to sections, still called $\psi_\alpha$, of $W^+$ near $p$.
By transversality at $p$ we have $\sum x_\a \psi_\a \neq 0$ for all  $x\neq 0$. Setting 
$$
A_\a:= \nabla_{e_\a}\A_\psi = \A_{\psi_\a}, 
$$
we see that
\begin{eqnarray}
\label{5}
x_\a A_\a: K_p =\ker({\cal A}_{\psi(p)}) \rightarrow \cok({\cal A}_{\psi(p)}) = \hat{K}_p
\end{eqnarray} 
is an isomorphism for every $x\in T_pX-\{0\}$.

\vspace{4mm}

The following technical  lemma ensures that $\A_\psi$ can be perturbed to satisfy the non-degeneracy assumption (\ref{normalrates}).

\begin{lemma}
We can modify $\psi$ without changing its zero set $Z(\psi)$ to ensure that $\{\psi_\a\}$ are orthonormal. 
\end{lemma}
\begin{proof}
This is proved in \cite{pr}, and is a special case of Lemma~\ref{lemma5.9} below.
 \end{proof}


Recall the matrices $M_\a = -e^\a_\cdot A_\a \in \End(K_p)$ and $\hat{M}_\a = - e^\a_\cdot A_\a^* \in \End(\hat{K}_p)$ from \eqref{IntroDefM}.  Let $K_\a^{\pm}$ and $\hat{K}_\a^{\pm}$ be the positive/negative eigenspaces of $M_\a$ and $\hat{M}_\a$ respectively. We are interested in describing their common positive eigenspaces 
$$
\K_p = \bigcap_\a K_\a^+ \quad \mbox{and}\quad  \hat{\K}_p = \bigcap_\a \hat{K}_\a^+.
$$

\begin{lemma}
\label{14}
The eigenvalues of $M_\a$ are $\lambda_\a = \pm 1$ and the corresponding eigenspaces can be described as
$$
K_\a^+\, =\, \span\Bigg\{\begin{pmatrix}e^\a_\cdot  b_\cdot \psi_\a\\ b\end{pmatrix}:\, b\in \Lambda^1X\Bigg\}\quad \text{and} \quad K_\a^-\, =\, \span\Bigg\{\begin{pmatrix} - e^\a_\cdot b_\cdot \psi_\a\\ b\end{pmatrix}:\, b\in \Lambda^1X\Bigg\}
$$
for every $\a$.
\end{lemma} 
\begin{proof}
By relation (\ref{cond}) and Corollary \ref{4}
$$
M_\a^2\, =\, e^\a_\cdot A_\a e^\a_\cdot A_\a\, =\, A_\a^*A_\a\, =\, Id
$$
i.e. $M_\a$ has eigenvalues $\pm 1$. Let now $\phi = \begin{pmatrix}\xi\\ b\end{pmatrix}\in K_p$ be a $\lambda_\a$- eigenvector of $M_\a$. Then 
\begin{eqnarray}
\label{30}
M_\a\phi\, =\, -e^\a_\cdot A_\a\phi\,=\, -\begin{pmatrix}-c(e^\a)&0\\ 0&-\hat{c}(e^\a)\end{pmatrix}\begin{pmatrix}b_\cdot\psi_\a\\ - P_{\psi_\a}^{ev*}\xi\end{pmatrix}\, = \, \lambda_\a\begin{pmatrix}\xi\\ b\end{pmatrix}.
\end{eqnarray}
By comparing the first rows of (\ref{30}) we see that
\begin{eqnarray}
\label{33}
\xi\, = \,\frac{1}{\lambda_\a}e^\a_\cdot b_\cdot\psi_\a\, =\, \lambda_\a e^\a\cdot b_\cdot\psi_\a 
\end{eqnarray}
since $\lambda_\a^2 = 1$.
It remains to show that given $b\in\Lambda^1X$, the above choice of $\xi$ gives equality of second rows of (\ref{30}). Using (\ref{3})
\begin{eqnarray*}
-\hat{c}(e^\a)P_{\psi_\a}^{ev*}\xi\, &=&\, (P_{\psi_\a}^{ev}\hat{c}(e^\a))^* \lambda_\a e^\a_\cdot b_{\cdot}\psi_\a\, =\, \lambda_\a (b_{\cdot}e^\a_\cdot P_{\psi_\a}^{ev}\hat{c}(e^\a))^*\psi_\a\\ &=& -\lambda_\a (P_{\psi_\a}^{ev}\hat{c}(b))^*\psi_\a\\ &=&\, \lambda_\a \hat{c}(b)P_{\psi_\a}^{ev*} \psi_\a.
\end{eqnarray*}
Also for every $\eta\in \Lambda^0X\oplus \Lambda^{2,+}X$ 
\begin{eqnarray*}
\Re\langle \eta,\, P_{\psi_\a}^{ev*}\psi_\a\rangle\, =\, \Re\langle \eta_\cdot \psi_\a,\, \psi_\a\rangle\, =\, \begin{cases} 0 &\text{if}\, \eta\in \Lambda^{2,+}X\\ \eta &\text{if}\,  \eta\in \Lambda^0 X  \end{cases}
\end{eqnarray*}
showing that $ P_{\psi_\a}^{ev*}\psi_\a = 1$. Hence
$$ 
\lambda_\a \hat{c}(b)P_{\psi_\a}^{ev*} \psi_\a\, =\, \lambda_\a \hat{c}(b)1\, =\, \lambda_\a b
$$
proving equality of the second rows of (\ref{30}). 
\end{proof}

It is elementary to see that the families $\{\hat{M}_\a\}$ and $\{M_\a\}$ are related as
\begin{eqnarray}
\label{13}
e^I_\cdot M_\a =- \hat{M}_\a e^I_\cdot \quad\text{if}\quad \a\in I\qquad \text{and}\qquad e^I_\cdot M_\a = \hat{M}_\a e^I_\cdot \quad\text{if}\quad \a\not\in I
\end{eqnarray}
for every string $I$ of odd length. It follows
$$ 
 e^I_\cdot\in \Hom(K_\a^{\pm},\,\hat{K}_\a^{\mp})\quad \text{if}\quad \a\in I \quad \text{and}\qquad e^I_\cdot  \in \Hom(K_\a^{\pm},\,\hat{K}_\a^{\pm})\quad \text{if}\quad \a\not\in I .
$$

\vspace{4mm}

\begin{lemma}
\label{34}
 The spaces $\K_p=\bigcap_\a K_\a^+$ and $\hat{\K}_p=\bigcap_\a\hat{K}_\a^+$ are at most one dimensional. If $\nabla \psi: T_pX\rightarrow W_p^+$ preserves orientation then $\K_p \simeq \R$  and $\hat{\K}_p =  \{0\}$. If $\nabla\psi$ reverses orientation then $\K_p= \{0\}$ and $\hat{K}_p\simeq\R$ is nontrivial.
\end{lemma}
\begin{proof}
Let $\phi\in\bigcap_\a K_\a^+$. Then for every even string $I$ and $\a\in I$
$$
M_\a e^I_\cdot\phi\, =\,  - e^I_\cdot \phi
$$
which implies that $e^I_\cdot \phi\in K_\a^-$. By Remark \ref{clbasis} $E_p = \span\{e^I\phi :\, I = \mbox{even}\}$  therefore $\bigcap_\a K_\a^+ = \langle\phi\rangle$ is at most one dimensional. The case with $\bigcap_\a \hat{K}_\a^+$ is analogous. 

Fix now a string $J$ and a nontrivial vector $\phi\in\Big(\bigcap_{\a\in J}K_\a^-\Big)\cap\Big(\bigcap_{\a\in J^c}K^+_\a\Big)$.
\begin{itemize}
\item $J$ is an even string if and only if 
$$
M_\a( e^J_\cdot\phi)\, =\, \begin{cases}e^J_\cdot M_\a\phi\, =\, e^J_\cdot \phi &\text{if}\,  \a\notin J\\ -e^J_\cdot M_\a\phi\,=\, e^J_\cdot \phi &\text{if}\,  \a\in J  \end{cases}
$$
for every $\a$ so that $\bigcap_\a K^+_\a = \langle e^J_\cdot\phi\rangle$. 
\item $J$ is an odd if and only if
$$
\hat{M}_\a( e^J_\cdot\phi)\, =\, \begin{cases}e^J_\cdot M_\a \phi\, =\, e^J_\cdot\phi &\text{if}\,  \a\notin J\\ -e^J_\cdot M_\a\phi\,=\, e^J_\cdot \phi &\text{if}\,  \a\in J  \end{cases}
$$
for every $\a$ so that $\bigcap_\a \hat{K}^+_\a = \langle e^J_\cdot \phi\rangle$.
\end{itemize}
This dichotomy shows also that either $\bigcap_\a K^+_\a$ or $\bigcap_\a \hat{K}^+_\a$ should be nontrivial at each zero of $\psi$. Say that $\a\sim \b$ iff $\a,\,\b\in J$ or $\a,\,\b\in J^c$. By Lemma \ref{14} if $\a\sim \b$ we can write $\phi = \begin{pmatrix}\pm e^\a_\cdot b_\cdot\psi_\a\\ b\end{pmatrix} = \begin{pmatrix}\pm e^\b_\cdot b_\cdot\psi_\b\\ b\end{pmatrix}$ for some common $b\in \Lambda^1X$ and if $\a\not\sim \b$ then $\phi = \begin{pmatrix}e^\a_\cdot b_\cdot\psi_\a\\ b\end{pmatrix} = \begin{pmatrix} - e^\b_\cdot b_\cdot\psi_\b\\ b\end{pmatrix}$ for the same $b\in \Lambda^1X$. In particular we have a description of the orthonormal basis $\{\psi_\a\}$ in terms of $\psi_1$ and $b$ as 
$$
\psi_\a\, =\,  \begin{cases}  b_\cdot e^\a_\cdot e^1_\cdot b_\cdot \psi_1 &\text{if}\, \a\sim 1 \\- b_\cdot e^\a_\cdot e^1_\cdot b_\cdot \psi_1  &\text{if}\, \a\not\sim 1\end{cases}.
$$
But $|J| + |J^c| =4$ hence $J,\, J^c$ are both even or both odd. Therefore $\{\psi_\a\}$ is positively oriented in $W^+_p$ for $J$ even and negatively oriented for $J$ odd.
\end{proof}

\begin{cor}
For the bundles $E$ and $F$ in \eqref{5.1A}, the index of $\D : \Gamma(E)\rightarrow \Gamma(F)$ is the second Chern number of the bundle $W^+$:
$$
\ind{\cal D}\, =\, c_2(W^+)[X].
$$

\end{cor}

\begin{proof}
As a consequence of Theorem \ref{Igor}  and Lemma~\ref{34}, the index of $\D$ is  the signed count of the zeros of $\psi$, which is given by the Euler number $e(W^+)[X]= c_2(W^+)[X]$.
\end{proof}

Of course, the index of $\D$ can also be calculated directly from the Atiyah-Singer Index Theorem; the result can then be seen to agree with $c_2(W^+)[X]$.   However, it is interesting to observe  that the index is the Euler number of a single bundle, and that this form emerges naturally from the concentration principle.
 
\vspace{5mm}

\noindent{\bf Example 6: \it ($J$-holomorphic curves in symplectic four-manifolds.)}

Recall the philosophy of Diagram~\ref{Ldiagram}:  if we can find a sub-bundle $\L$ of $\Hom(E,F)$ whose sections satisfy the concentration condition, then we obtain concentrating operators $D_s$ with  singular sets $Z_\A$ of possibly different dimensions.   This example illustrates this phenomenon in dimension four, by showing how a sub-bundle $\L$ can be constructed from    a symplectic structure.  

Let $(X^4, \omega)$ a closed symplectic manifold with a complex hermitian line bundle $L$ and a section $\psi \in\Gamma(L)$ whose zero set  is a transverse disjoint union $Z_\psi = \cup_\ell Z_\ell$ of  symplectic submanifolds of $X$.  Let  $N_\ell$ be the  symplectic normal bundle of $Z_\ell$.  Choose an almost complex structure $J$ and a Riemannian metric $g$ on $X$ so that   $(\omega, J, g)$ is a compatible triple and each $Z_\ell$ is $J$-holomorphic;   the symplectic and metric normal bundles of $Z_\ell$ are the same.   

As usual, write $TX\otimes \mathbb{C} = T^{1,0}X\oplus T^{0,1}X$ and define the  canonical bundle to be the complex line bundle $K^X={\Lambda}^{2,0}X$.   In this context, the direct sum $W=W^+\oplus W^-$ of the complex rank 2 bundles
$$
W^+ =   L\oplus L\bar{K}, \qquad W^- = L\otimes {\Lambda}^{0,1}X
$$
has a $\Spin^c$ structure and a Clifford multiplication $T^*X\otimes  W\to W$  from wedging and contracting $(0,1)$ forms as in formula \eqref{interior/exteriormult}. 

Let $\nabla^L$ be a hermitan connection on $L$ and $\nabla^X$ the Levi-Civita connection on $X$. These can be used to build a $\Spin^c$ connection $\nabla = \nabla^L\oplus\nabla^{L\bar{K}}$ on $W^+$. There is also the projection to the $(0,1)$ part of $T^*X$ of $\nabla^L$ namely $\del_L \psi: = \frac{1}{2}(\nabla^L\psi + i\nabla^L\psi \circ J)$. Then a Dirac operator is defined by  
$$
D = \sqrt{2}(\del_L + \del_L^*) : L\oplus \bar{K}L \rightarrow {\Lambda}^{0,1}X\otimes L.
$$

We would like to study the perturbed operator $\D + s\A_\psi$. Fix one component  $Z=Z_\ell$ with normal bundle $N=N_\ell$. By the transversality of $\psi$,  the map 
$\nabla^L\psi : N \rightarrow L|_Z$ is an $\mathbb{R}$-linear isomorphism.

In order for the nondegeneracy condition \eqref{normalrates}  to hold we need the following:
\begin{lemma}
\label{lemma5.9}
We can change $\psi$ without changing its zero set so that $\nabla^L\psi : N\rightarrow L|_Z$ becomes orthogonal.
\end{lemma}
\begin{proof}
We consider the bundles $N$ with the induced metric and $L|_Z$ as a real vector bundle with the induced metric $h$ from the hermitian metric. Let $O(N, L|_Z) = \{ H\in \Hom(N, L|_Z) | H^*h = g\}$, a deformation retraction of $\Hom(N , L|_Z)$. Therefore there is a smooth path of bundle maps $[0,1]\ni t\rightarrow H_t\in \Hom(N,\,L|_Z)$ so that $H_0 = \nabla^L\psi$ and $H_1 \in O(N, L|_Z)$. This path can be chosen so that $H_t$ is invertible for every $t\in [0,1]$. As a consequence there exist constant $C>0$ such that 
$$
\inf_{(t, v)\in [0,1]\times S}\Big|H_t(v)\Big|\, >\, C
$$
where $S$ is the unit sphere bundle of the normal bundle $N$. 

Now use the exponential map on the normal bundle $N$ of $Z$ to define a tubular neigbroorhood $\N$, a parallel transport map $\tau : L|_Z \rightarrow L|_\N$ along normal geodesics and set $x = \exp(v)$. Let $B(Z,R)\subset \N$ and $\rho$ be a smooth cutoff function with $\text{supp}(\rho)\subset B(Z,2R)$ and $\rho|_{B(Z,R)} \equiv1$. We redefine $\psi$ in $B(Z,2R)$ as
$$
\Psi(v): =\,\psi(v)\,+\, \tau(H_{\rho(v)}(v)\,-\, \nabla^L_v\psi|_Z).
$$
Note that
$$
\Big|\psi(v) - \, \tau\nabla^L_v\psi|_Z\Big| = O(|v|^2)\leq C_1 |v|^2
$$ 
for $v\in \N$. Clearly $\Psi$ has also transverse intersection with the zero section at $Z$ and satisfies the conclusion of the lemma at $Z$. Choosing $R<\frac{C}{2C_1}$ we get that for $v\in B(Z,2R)\backslash Z$
$$ 
\frac{|O(|v|^2)|}{|v|}\,\leq\, 2C_1R\,<\, C\, <\, \left|\tau H_{\rho(v)}\left(\frac{v}{|v|}\right)\right|
$$
therefore there are no other zeros of $\Psi$ in $B(Z,2R)$ except at $Z$. Repeating this proccess for each component of the singular set of the original $\psi$ we are done.
\end{proof}

Fix now $p\in Z$ and local coordinates $\{x_i\}$ in $X$ so that $Z = \{x_1 = x_2 =0\}$ and orthonormal frames $\{e_1,e_2 = J(e_1)\}$ and $\{e_3,e_4 = J(e_3)\}$ trivializing $N$ and $TZ$ respectively around $p$. By Lemma \ref{lemma5.9}, the set $\{\psi_i = \nabla^L_{e_i}\psi\}_{i=1,2}$ is an orthonormal frame trivializing $L|_Z$ around $p$ and we extend it to the normal directions to a frame trivializing $L$. Then $\psi$ expands in the normal directions of $Z$ as 
$$ 
\psi = x_1\psi_1 + x_2\psi_2 + O(|x|^2).
$$

Denote $(\nabla_{e_i}\A)\psi = \A_{\nabla^L_i\psi} = A_i$. 

We now have to consider the matrices $M_\a = -e^\a_\cdot A_\a, \,\a =1,2$ and their common positive spectrum.
As in Lemma \ref{14}, the positive eigenspaces are given by
$$
K^+_\a =  \span\Bigg\{\begin{pmatrix}e^\a_\cdot  b_\cdot \psi_\a\\ b\end{pmatrix}:\, b\in \Lambda^1X\Bigg\}, \  \a=1,2.
$$
There are two cases: 

\noindent $\bullet$\ \  The map $\nabla^L\psi : N\rightarrow L|_Z$ preserves the natural orientations as an $\mathbb{R}$ linear map. Then 
$$
e_2 = J(e_1)\quad \mbox{and}\quad \nabla^L_{J(e_1)}\psi = \psi_2 = i\psi_1 = i\nabla^L_{e_1}\psi
$$
so that $ e^1_\cdot\psi_1 + e^2_\cdot \psi_2 = D\psi|_Z = \del_L\psi|_Z = 0$. Then
\begin{eqnarray}
\K = K^+_1 \cap K^+_2 = \span\Bigg\{\begin{pmatrix} \psi_1\\ e^1\end{pmatrix}, \begin{pmatrix} \psi_2\\ e^2\end{pmatrix}\Bigg\} = \Bigg\{\begin{pmatrix}\nabla^L_{v^\sharp}\psi\\ v\end{pmatrix}:\, v\in N^*\Bigg\} \simeq N^*.
\end{eqnarray}
Also
$$
\hat{\K} = \span\{ e^3_\cdot \K, e^4_\cdot \K\} \simeq \ \mbox{orthogonal complement of $\omega$ in} \,{\Lambda}^{2,+}X \simeq K^X|_Z \simeq T^*Z\otimes N^*.
$$
Hence a local operator can be defined using the tangent frame $\{e_1,e_2\}$ as 
$$
D^Z = e^1_\cdot \nabla^L_{e_1} + e^2_\cdot \nabla^L_{e_2} =  \bar{\partial}_{N^*} : \Gamma(N^*)\rightarrow \Gamma(T^*Z_i\otimes N^*) 
$$
and by Riemann-Roch 
$$
\ind D^Z = 2 N^2 - 2(g - 1) = 2(L|_Z)^2 - 2(g-1)\quad \mbox{where}\ g =\ \mbox{genus of}\ Z.
$$
Also since $\nabla^L\psi|_Z : N\rightarrow L|_Z$ preserves orientation, the adjunction formula applies to give 
$$
2(g- 1) = (L|_Z)^2 +  L|_Z K. 
$$
Hence $ \ind D^Z  =  (L|_Z)^2 - K|_Z\cdot L|_Z$ in this case.

\begin{figure}
\vspace{.2cm}
\begin{center}
\psfrag{X}{$X$}
\psfrag{A}{$Z_1$}
\psfrag{B}{$Z_2$}
\psfrag{C}{$Z_3$}
\includegraphics{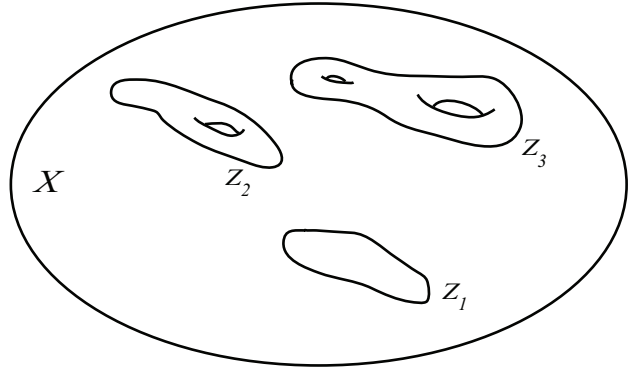}
\end{center}
\caption{The zero set of $\psi$}
\end{figure}
\vspace{.2cm}

\noindent $\bullet$\ \  If $\nabla^L\psi|_Z$ reverses orientation then adjunction formula gives
$$
2(g - 1) = (\bar{L}|_Z)^2 + K\bar{L}_Z = (L|_Z)^2 - K\cdot L|_Z
$$ 
and a similar calculation computes the local operator in this case $ D^Z = \bar{\partial}^* : T^*Z \rightarrow \mathbb{C}$.  By Riemann-Roch we then have
$$
\ind D^Z = 2(g - 1)= (L|Z)^2 - K\cdot L|_Z.
$$
Summing the contributions of the local indices from all the components $Z_\ell$ with $L_\ell = L|_{Z_\ell}$ we get 
$$
\ind D = c_2(W^+)[X] =  L^2 - KL = \sum_\ell( L_\ell^2 - K\cdot L_\ell) = \sum_\ell \ind D_\ell.
$$
This is a familiar formula in SW theory. It describes the dimension of the SW moduli space in terms of the bundles $K$ and $L$. It also suggest that the localization formula in Theorem \ref{Igor} can be generalized to calculate the index of $D$ by summing up contributions from indexes of Dirac operators defined from the local information of the bundles $E$ and $F$ and the first order jets of  $\A$  along the  components $Z_\ell$ of the singular set $Z_\A$. 
\medskip
   
\vspace{1cm}

\section{Twisted Examples} 
\label{sec6}
\bigskip

Examples 7 and 8 below are twisted versions of Examples~3 and 5. We use an $SU(2)$ bundle to twist the spinor bundles and the associated Dirac operators. Similar constructions can be realized for any  compact Lie group. 
\vspace{10mm}

\noindent{\bf Example 7: \it Spinor-form pairs twisted by $SU(2)$-bundles.}
\bigskip

Let $(X, g , W^{\pm}, c)$ be an even-dimensional closed  Riemannian manifold with a $\Spin^c$ structure,  and let  $(E, h)$ be a Hermitian $SU(2)$-bundle on $X$. Set also $\mathfrak{su}(E): = \{A\in \End_\mathbb{C}(E) : A + A^* = 0,\, \tr_\mathbb{C}A=0\}$ where $A^*$ is the Hermitian adjoint of $A$. Differences of Hermitian connections on $E$ are sections of ${\Lambda}^1X\otimes \mathfrak{su}(E)$. Equip $E$ with a Hermitian connection $\nabla_E$ and $W$ with a $\Spin^c$ connection;  these induce connections $\nabla^{W\otimes E}$ on $W\otimes E$ and $\nabla$ on $\Lambda^* X\otimes \mathfrak{su}(E)$. The symbol maps $c$ and $\hat{c}$ extend as 
$$
c(v)\otimes Id_E : W^+\otimes E \rightarrow W^-\otimes E \quad\mbox{and}\quad \hat{c}(v)\otimes Id_{\mathfrak{su}(E)} : {\Lambda}^{odd}X\otimes\mathfrak{su}(E)\rightarrow {\Lambda}^{ev}X\otimes \mathfrak{su}(E).
$$
Finally we get operators
$$
D_E = (c\otimes id_E)\circ \nabla^{W\otimes E} \quad \mbox{and}\quad d_E + d_E^* = (\hat{c}\otimes id_{\mathfrak{su}(E)})\circ \nabla.
$$

We define a Clifford multiplication $c_E$ to include $\End(E)$-valued forms by 
\begin{eqnarray}
\label{9}
c_E : {\Lambda}^*(X)\otimes \End(E)&\rightarrow & \End(W\otimes E)\\ \nonumber \eta\otimes A&\mapsto& \eta_\cdot \otimes A
\end{eqnarray}

\noindent The restriction of $c_E$ to the subspace ${\Lambda}^*(X)\otimes \mathfrak{su}(E)$ defines maps  
$$
P^{ev}: \ W^+\otimes E \rightarrow {\Hom}_{\mathbb{C}}\Big({\Lambda}^{ev}(X)\otimes \mathfrak{su}(E),\,W^+\otimes E\Big) 
$$
and
$$
P^{odd}:\  W^+\otimes E \rightarrow {\Hom}_{\mathbb{C}}\Big({\Lambda}^{odd}(X)\otimes \mathfrak{su}(E),\,W^-\otimes E\Big) 
$$
both given by $\psi\otimes e \mapsto  c_E(\cdot)\psi\otimes e$.

\begin{prop}
 For fixed $\Psi \in \Gamma(W^+\otimes E)$,  the perturbed operator
$$
{\cal D}_s= {\cal D} + s\,{\cal A}_{\Psi}:\Gamma(W^+\otimes E)\oplus\Omega^{odd}(X,\mathfrak{su}(E))\rightarrow\Gamma(W^-\otimes E)\oplus\Omega^{ev} (X,\mathfrak{su}(E))
$$
with
\begin{eqnarray*}
{\cal D}\ =\ \begin{pmatrix}D_E&0\\0&d_E+ d_E^*\end{pmatrix}
\qquad\mbox{and}\qquad
{\cal A}_{\Psi}\, = \, \begin{pmatrix}0&P^{odd}_{\Psi}\\- P^{ev*}_{\Psi}&0\end{pmatrix}
\end{eqnarray*}
satisfies the concentration relation (\ref{cond}). Here $P^{ev*}_{\Psi}$ denotes the adjoint of $ P_{\Psi}^{ev}$. 
\end{prop}
\begin{proof}
By linearity, it suffices to show (\ref{cond}) for $\Psi$ of the form $\psi\otimes e$. The symbol of ${\cal D}$, applied to a covector $v$, and its adjoint are given by
$$
{\sigma_{\cal D}(v)}\, =\, \begin{pmatrix}c(v)\otimes Id_E&0\\0& \hat{c}(v)\otimes Id_{\mathfrak{su}(E)}\end{pmatrix} 
$$
and
$${ \sigma_{\cal D}(v)^*}\ =\ \begin{pmatrix}-c(v)\otimes Id_E&0 \\0& -\hat{c}(v)\otimes Id_{\mathfrak{su}(E)}\end{pmatrix}.
$$
Checking the concentration relation amounts to proving the identity 
$$
P_{\psi\otimes e}^{ev}\circ (\hat{c}(v)\otimes Id_{\mathfrak{su}(E)})\, =\, (c(v)\otimes Id_E)\circ P_{\psi\otimes e}^{odd}.
$$
By linearity, it is enough to check that the identity holds when applied to $b\otimes B \in {\Lambda}^{odd}X\otimes \mathfrak{su}(E)$. Then 
$$
(c(v)\otimes Id_E)\circ P_{\psi\otimes e}^{odd}(b\otimes B) = (c(v)\otimes Id_E) (c(b)\psi \otimes B(e)) =  (c(v)c(b) \psi)\otimes B(e)
$$
and
\begin{eqnarray*}
P_{\psi\otimes e}^{ev}\circ (\hat{c}(v)\otimes Id_{\mathfrak{su}(E)})(b\otimes B) &=& P_{\psi\otimes e}^{ev}(\hat{c}(v)b\otimes B) = (c(\hat{c}(v)b)\psi) \otimes B(e)\\ &=& (c(v)c(b) \psi)\otimes B(e),
\end{eqnarray*}
where the last equality holds by $(\ref{clhat})$.
\end{proof}

\vspace{5mm}

We conclude with a generalization of Witten's  perturbation of $d+d^*$ described in Example~3.\\

\noindent{\bf Example 8: \it Witten's deformations with values in bundles.}
\bigskip

Let $(X, g)$ closed even-dimensional Riemannian.  Fix  a Hermitian bundle  $(E, h)$ on $X$ with  a Hermitian connection $\nabla_E$, and let
$$
\Sym(E): = \{A\in \End_\mathbb{C}(E) : A = A^* \}
$$
be the subspace of self-adjoint endomorphisms with the inner product $\langle A, B\rangle = \tr(AB)$.  
Recall the Clifford representations $\hat{c}$ and $\tilde{c}$ on $\Lambda^*(X)$ described in  equation (\ref{hatcl}) of Example~3.  Then $\hat{c}$ extend  to the map
$$
\sigma(v) = \hat{c}(v)\otimes Id_{\End(E)} : \Lambda^{odd}(X)\otimes\End(E)\, \rightarrow\, \Lambda^{ev}(X)\otimes \End(E).
$$
For $\Psi=\a\otimes A\in \Lambda^1(X)\otimes \Sym(E)$, define
$$
\begin{array}{cccc}
\A_{\a\otimes A}: &\Lambda^{odd}(X)\otimes\End(E) & \rightarrow & \Lambda^{ev}(X)\otimes \End(E)\\ 
 & \b\otimes B& \mapsto & \tilde{c}(\a)\b \otimes A\circ B
\end{array}
$$
By linearity, we can extend this to a map $\A_\Psi$ for every $\Psi \in  \Lambda^1(X)\otimes \Sym(E)$. 
Finally, the induced connection $\nabla$ on $\Lambda^*(X)\otimes \End(E)$ gives an operator
$$
D_E = \sigma\circ \nabla: \Gamma(\Lambda^{odd}(X)\otimes\End(E))\, \rightarrow\, \Gamma(\Lambda^{ev}(X)\otimes\End(E)).
$$
\begin{prop}
For fixed $\Psi \in  \Gamma(\Lambda^1(X)\otimes \Sym(E))$,  the perturbed operator
$$
D_s= D_E + s\,{\cal A}_\Psi
$$
satisfies the concentration condition (\ref{cond}).
\end{prop}
\begin{proof}
Again, it is enough to verify (\ref{cond}) for $\Psi = \b\otimes B \in {\Lambda}^{odd}X\otimes \End(E)$. Note that 
\begin{eqnarray*}
(\hat{c}(v)^*\otimes Id_{\End(E)})\circ {\cal A}_{\a\otimes A} (\b\otimes B) &=& (\hat{c}(v)^*\otimes Id_{\End(E)} (\tilde{c}(\a)\b \otimes A B)\\ &=&  \hat{c}(v)^*\tilde{c}(\a) \b\otimes A B 
\end{eqnarray*}
and
\begin{eqnarray*}
\A^*_{\a\otimes A}\circ (\hat{c}(v)\otimes Id_{\End(E)})(\b\otimes B) &=& \A^*_{\a\otimes A}(\hat{c}(v)\b\otimes B)\\ &=& \tilde{c}(\a)^*\hat{c}(v)\b\otimes A^* B.  
\end{eqnarray*}
These are equal because $A^* = A$ and $(\ref{twocliff})$ holds.
\end{proof}

 \vspace{1cm}

\end{document}